\documentclass[reqno,12pt]{amsart} 
\usepackage{mathrsfs}
\usepackage[margin=1in]{geometry}
\usepackage{amssymb,amsmath,amsthm,amsfonts}
\usepackage{latexsym}
\usepackage{amssymb}
\usepackage{amsmath}
\usepackage{amsfonts}
\usepackage{amsthm}
\usepackage{amscd}
\usepackage{texdraw}
\usepackage{color}
\usepackage{cite}
\usepackage[centertags]{amsmath}
\usepackage[colorlinks=true,urlcolor=blue,linkcolor=red,citecolor=red]{hyperref}
\usepackage{cleveref}

\newlength{\defbaselineskip}
\newcommand{\setlinespacing}[1]%
           {\setlength{\baselineskip}{#1 \defbaselineskip}}

\theoremstyle{plain}
\theoremstyle{plain}
\newtheorem{definition}{Definition}[section]
\newtheorem{theorem}[definition]{Theorem}
\newtheorem{lemma}[definition]{Lemma}
\newtheorem{proposition}[definition]{Proposition}
\newtheorem{corollary}[definition]{Corollary}
\newtheorem{remark}[definition]{Remark}

\allowdisplaybreaks
\begin{document}

\title[Wiener type theorems for countable limits of quasi-Beurling algebras]{Wiener type theorems for countable limits of quasi-Beurling algebras and maximizing results on weights}

\author[P. A. Dabhi]{Prakash A. Dabhi}
\address{Department of Basic Sciences, Institute of Infrastructure Technology Research and Management (IITRAM), Maninagar (East), Ahmedabad - 380026, Gujarat, India}
\email{lightatinfinite@gmail.com, prakashdabhi@iitram.ac.in}

\author[K. B. Solanki]{Karishman B. Solanki}
\address{Department of Mathematics, Indian Institute of Technology Ropar, Rupnagar, Punjab - 140001, India}
\email{karishsolanki002@gmail.com, staff.karishman.solanki@iitrpr.ac.in}
%\thanks{}



%%% ----------------------------------------------------------------------
\begin{abstract}
We establish the vector-valued Wiener type theorems for countable projective and inductive limits of quasi-Banach algebras in a weighted setting for both finite and infinite dimensional cases. As an application, we extend the notions of rapidly decreasing and exponentially decreasing sequence spaces using quasi-Beurling algebras and show that they are inverse-closed; and obtain a hierarchy of inverse-closed vector-valued algebras using weights. In addition, we derive maximizing results on weights for the nonadmissible weighted version of Wiener's theorem in both discrete and continuous cases.
\end{abstract}
%%% ----------------------------------------------------------------------
\maketitle
\tableofcontents
%%% ----------------------------------------------------------------------

%---------------------------------------------------------------------------
\section{Introduction}
As indicated by the title, the study presented here is for the weighted version of Wiener's classical result \cite{wi} which states that if a continuous function $F:\mathbb{T}\to\mathbb{C}$, where $\mathbb{T}$ is the unit circle, is nowhere zero and has an absolutely convergent Fourier series, then its inverse $\frac{1}{F}$ also has an absolutely convergent Fourier series. This well known result is referred to as \emph{Wiener's theorem}. The convolutive Banach algebra $\ell^1(\mathbb{Z})$ of all absolutely summable sequences on $\mathbb{Z}$ is isomorphic, via the Fourier transform, to the algebra of all continuous functions on $\mathbb{T}$ which have absolutely convergent Fourier series. Thus, Wiener's theorem can be restated as follows: an element $f\in\ell^1(\mathbb{Z})$ is invertible if and only if its Fourier transform $\widehat{f}:\mathbb{T}\to \mathbb{C}$ is nowhere zero. In \cite{Ze}, \.Zelazko obtained its $p$-power analogue for $0<p<1$. Its weighted versions are also investigated and in this direction Domar \cite{Do} showed that the same follows for the weighted algebra $\ell^1_\omega(\mathbb{Z})=\{\{a_n\}_{n\in\mathbb{Z}}:\sum_{n\in\mathbb{Z}}|a_n|\omega(n)<\infty\}$ provided $\omega$ is a non-quasianalytic weight, that is, $\sum_{n\in\mathbb{Z}} \frac{\log \omega(n)}{1+n^2}<\infty$. Note that non-quasianalytic weights are admissible. A weight $\omega$ on $\mathbb{Z}^d$ is \textit{admissible} if it satisfies the \emph{Gel'fand–Ra\v{i}kov–\v{S}hilov (GRS) condition} \cite{GRS}: $\lim_{n\to\infty} \omega(nk)^\frac{1}{n}=1$ for all $k\in\mathbb{Z}^d$; and it is known that Wiener's theorem is true for admissible weights. In fact, the result is two way even in higher dimension and is stated below, see \cite[Theorem 5.24 and Corollary 5.27]{Groadmissible} for proof. A subalgebra $\mathcal{X}$ of $\ell^1(\mathbb{Z}^d)$ is said to be \textit{inverse-closed} if every $f\in\mathcal{X}$ that is invertible in $\ell^1(\mathbb{Z}^d)$ is also invertible in $\mathcal{X}$. 

\begin{theorem} [Wiener-Domar-\.Zelazko-GRS] \label{thm:Wiener-Domar-Zelazko}
For $d\in\mathbb{N}$ and $0<p\leq1$, $\ell^p_\omega(\mathbb{Z}^d)$ is inverse-closed if and only if $\omega$ is an admissible weight on $\mathbb{Z}^d$.
\end{theorem}
In \cite{Bh}, Bhatt and Dedania showed that if $\omega$ is not admissible and $f\in\ell^1_\omega(\mathbb{Z})$ is such that its Fourier transform vanishes nowhere, then $f$ may not be invertible in $\ell^1_\omega(\mathbb{Z})$ but there exists a weight $\nu$ closely related to $\omega$ such that the inverse of $f$ is in $\ell^1_\nu(\mathbb{Z})$. Its $p$-power analogue is obtained in \cite{De} for $0<p<1$, and in \cite{Da} for $1<p<\infty$ putting some conditions on weight $\omega$. The two variable analogue is obtained in \cite{Da2} for $0<p\leq1$. The infinite variable analogue for $p=1$ and without weight is studied by Zagorodnyuk and Mitrofanov in \cite{Zag}. 

All of the above results were concerned with complex valued functions. Since operator and Banach space valued functions are of utmost interest in analysis, vector-valued analogues of Wiener's theorem have also been investigated in literature. Bochner and Phillips, in \cite{Bo}, generalized Wiener's theorem for Banach algebra valued functions and also showed its continuous analogue on the real line. In \cite{kb}, the Banach algebra valued $p$-power weighted analogues of Wiener's theorem for $0<p<\infty$ and its continuous analogue for $1\leq p<\infty$ are obtained. For $0<p\leq1$, the multivariate, both finite and infinite, $p$-power analogues of it are established in \cite{kb2}. Note that the results in \cite{kb,kb2} cover all the results mentioned above and are stated in Section~\ref{subsec:known analogues of Wiener's theorem}. It is important to mention here that these results were established using the classical Banach algebra techniques developed by Gel'fand et. al., we refer the reader to \cite{GRS, Kan} for an exposition.

Weights are very useful in analysis as they are beneficial in many situations, for example, to quantify the growth and decay conditions. Specifically, weights are used to control or ignore oscillations of a function at infinity to analyze its asymptotic behavior or its behavior around a certain point. The most notable examples are the polynomial weights $\omega_s(n)=(1+|n|^2)^\frac{s}{2}$ and the subexponential weights $\omega(n)=e^{|n|^\gamma}$, for $\gamma\in(0,1)$, on $\mathbb{Z}^d$, as well as their extensions to $\mathbb{R}^d$. These weights are widely used in analysis, particularly in time-frequency analysis and in the study of linear and nonlinear PDEs; see, for example, \cite{GroWeight2007, Groadmissible, grafakos2008modern, reich2015nonanalytic, PelinovskyStefanov, Cuccagna} and references therein.

The vector-valued and weighted analogues of Wiener type theorems have been studied extensively in the literature, refer to \cite{Bo, De, Da, Da2, Bh, Krishtal2011, kb, kb2, Zag, Ze} and references therein; of course, the list is not exhaustive. The reason is that it has many applications in various fields of analysis such as signal theory, frame theory, sampling theory, Gabor analysis, etc.; see \cite{Badi,Ba,kb3,Gro,Gr, Bal25,FeiWeight1979, GroWeight2007, Ku, Fageotcountablelimits, Sun2007, Sun2011, SUN2005567, ShinSun2013} and references therein; the list is by no means exhaustive.

Here, we are interested in Wiener type theorems for the countable projective and inductive limits of Banach algebra valued quasi-Banach algebras in weighted setting. Recently, such spaces were studied in weighted setting for admissible weights by Fern\'andez, Galbis, and Toft in \cite{FernandezGalbisToft2014,FernandezGalbisToft2015}.  In \cite{Fageotcountablelimits}, Fageot et al. studied the inverse-closedness of nuclear algebras (algebras without the norm structure): the spaces of rapidly and exponentially decreasing sequences, using algebras of the following types: 
\begin{itemize}
    \item Type-I: countable projective limit of Banach spaces  $\bigcap_{n\in\mathbb{N}}\ell^1_{\omega_n}(\mathbb{Z}^d),$ where $\{\omega_n\}_{n\in\mathbb{N}}$ is a family of weights such that $\omega_n\leq\omega_{n+1}$;
    
    \item Type-II: countable inductive limit of Banach spaces $\bigcup_{n\in\mathbb N} \ell^1_{\omega_n}(\mathbb Z^d)$, where $\{\omega_n\}_{n\in\mathbb{N}}$ is a family of weights such that $\omega_{n+1}\leq\omega_n$.
\end{itemize}
%Nuclear spaces were first introduced in \cite{Grothendieck1955} to better understand aspects of functional analysis beyond Banach spaces. 
The projective and inductive limits of Banach spaces are important for the study of infinite matrices dominated by convolution kernels, particularly in the field of signal and image processing; and in the areas of approximation and sampling theory. The quasi-normed spaces also have several applications in the fields of  coorbit theory and time-frequency analysis, we refer the reader to \cite{Rauhut2007Wiener, Rauhut2007Coorbit, grafakos2008modern, BastianoniCordero2022, GrochenigPfeufferToft2024} and references therein.  This, together with the importance of weights, vector-valued sequences, and various versions of Wiener’s theorem in analysis as mentioned earlier, motivate the present study.

From Theorem~\ref{thm:Wiener-Domar-Zelazko}, it follows that an algebra of Type-I is inverse-closed if and only if each weight $\omega_n$ satisfies the GRS condition. The study in \cite{Fageotcountablelimits} treated the family of weights $\{\omega_n\}_{n\in\mathbb{N}}$ for algebra of Type-II with each $\omega_n$ not satisfying the GRS condition, but imposed the extended GRS condition to remain within the realm of inverse-closedness. A family of weights $\{\omega_n\}_{n\in\mathbb{N}}$ on $\mathbb{Z}^d$ satisfies the \textit{extended GRS condition} \cite{FernandezGalbisToft2014} if
\begin{align} \label{def:extended GRS condition}
    \inf_{n\in\mathbb{N}} \left( \lim_{k\to\infty} \omega_n(km)^\frac{1}{k} \right) =1 \quad \text{for all} \quad m\in\mathbb{Z}^d.
\end{align}   
In particular, they proved the following result.
\begin{theorem} \cite[Theorem 20]{Fageotcountablelimits} \label{thm:Fageot et. al. Type-II}
An algebra of Type-II: $\bigcup_{n\in\mathbb N} \ell^1_{\omega_n}(\mathbb Z^d)$ is inverse-closed if and only if the family of weights $\{\omega_n\}_{n\in\mathbb{N}}$ on $\mathbb{Z}^d$ satisfies the extended GRS condition.
\end{theorem}
Here, we go beyond this framework by not imposing the GRS or the extended GRS conditions. As stated earlier, the vector-valued weighted analogues of Wiener's theorem even for nonadmissible weights are known. Here, we prove similar results for algebras of Type-I and Type-II, where the weights need not satisfy the GRS or extended GRS condition. Also, we extend the analysis to the infinite dimensional case, a scenario not examined in \cite{Fageotcountablelimits}; and consider even more generalized form of these algebras as per the definition given below.

For the infinite dimensional case, we use the set of all finitely supported sequences of integers over $\mathbb{N}$, denoted by $\mathbb{Z}^\mathbb{N}$, that is, an element $\alpha$ of $\mathbb{Z}^\mathbb{N}$ is of the form 
\begin{align*}
    \alpha=(\alpha_1,\alpha_2,\dots,\alpha_n,0,0,0,\dots), \quad \text{where} \quad n\in\mathbb{N} \quad \text{and} \quad \alpha_i\in\mathbb{Z} \quad (1\leq i \leq n).
\end{align*}

\begin{definition} \label{def:Type-I and Type-II}
Let $\mathcal{A}$ be a unital Banach algebra. For $\mathbb{X}\in\{\mathbb{Z}^d,\mathbb{Z}^\mathbb{N}\}$, where $d\in\mathbb{N}$, we define algebras of 
\begin{itemize}
    \item Type-I: countable projective limit of (quasi-)Banach algebras 
    $$\bigcap_{n\in\mathbb{N}}\ell^{p_n}_{\omega_n}(\mathbb{X},\mathcal{A}),$$ 
    where $\{p_n\}_{n\in\mathbb{N}}$ is a decreasing sequence in $(0,1]$ and $\{\omega_n\}_{n\in\mathbb{N}}$ is a family of weights on $\mathbb{X}$ such that $\omega_n\leq\omega_{n+1}$ for all $n\in\mathbb{N}$.

    \item Type-II: countable inductive limit of (quasi-)Banach algebras
    $$\bigcup_{n\in\mathbb{N}}\ell^{p_n}_{\omega_n}(\mathbb{X},\mathcal{A}),$$ 
    where $\{p_n\}_{n\in\mathbb{N}}$ is an increasing sequence in $(0,1]$ and $\{\omega_n\}_{n\in\mathbb{N}}$ is a family of weights on $\mathbb{X}$ such that $\omega_n\geq\omega_{n+1}$ for all $n\in\mathbb{N}$.
\end{itemize}
\end{definition}

The above definition is justified by Lemma~\ref{lem:weighted Fourier algebras inclusions}. To the best of authors' knowledge, these spaces have not been studied in the existing literature. As we require only the algebraic structure, we do not discuss the topological aspects of these algebras. However, it is worth noting that there is no (quasi-)norm structure on the above algebras, and that the quasi-normed spaces appearing in the union and intersection are not locally convex, unlike the normed spaces typically considered in standard definition. But these spaces possess well-defined topologies induced by the corresponding quasi-norms in place of norms. Moreover, the algebra $\mathcal{A}$ need not be commutative, placing us in the realm of noncommutative algebras. This makes this investigation more interesting. The algebras of Type-I and Type-II are also referred as \textit{Fr\'echet algebra} and \textit{LF-algebra}, respectively.

\begin{remark} We consider $\{p_n\}$ to be a monotone sequence, not necessarily strictly monotone. For instance; taking $p_n=1$ and $\mathcal{A}=\mathbb{C}$, we recover the spaces considered in \cite{Fageotcountablelimits,FernandezGalbisToft2014,FernandezGalbisToft2015}. The following remarks highlight some useful observations regarding the above definition.
\begin{enumerate}
    \item If we take $p_n\in(0,1)$, these algebras can be seen as ``perfect quasi-spaces'' in the sense that each member of the limits are quasi-Beurling algebras. 
    
    \item Taking $\{p_n\}$ to be strictly monotone allows one to consider the case where all weights are equal, including the trivial weight, and even then the algebras are well defined, which was not the case earlier. Similarly, by choosing the family of distinct weights, one can fix $p_n=p\in(0,1]$.
    
    \item If $\{p_n\}$ for an algebra of Type-I is such that $\lim_{n\to\infty}p_n\to0$, then it would contain only sequences with finite support; hence, such cases shall be avoided.
    
    \item It is worth highlighting that the case where all weights are not admissible is particularly important, and addressing this scenario adds to the interest and novelty of the present study. In fact, it is the highlight of the work presented here.
\end{enumerate}
\end{remark}

Further, for $\mathbb{X}=\mathbb{Z}$, analogous results for the non-quasi case, that is, when $\{p_n\}_{n\in\mathbb{N}}$ is a monotone sequence in $(1,\infty)$ with all other assumptions as is, are obtained under suitable additional conditions on the weights in the associated family (Theorem~\ref{thm:p>1, proj limit d=1} and Theorem~\ref{thm:p>1 inductive limit d=1}). These results are presented at the end of the paper so as not to disrupt the overall flow.

The results presented here are expected to have applications in the relevant areas of analysis mentioned above. As an immediate application, we extend our work to the vector-valued spaces of rapidly decreasing and exponentially decreasing sequences by defining them as projective and inductive limits of quasi-normed spaces in place of normed spaces and establish their inverse-closedness. We note that the weights used in the definition of these spaces are polynomial and exponential. Considering the importance of these weights along the subexponential weights in various areas of mathematics, we obtain a chain of inclusion showing the hierarchy of inverse-closed vector-valued algebras in quasi-Beurling framework.

%%%%%%%%%%%%%%%%%%%%%%%%%%%%%%%%%%%%%%%%%%%%%%%%%%%%%
The second part of the paper concerns the existence of the smallest and largest ($p$-Beurling) algebras between $\ell^p_\omega(\mathbb{Z},\mathcal{A}) \subset \ell^1(\mathbb{Z},\mathcal{A})$ that contain an inverse of $f\in\ell^p_\omega(\mathbb{Z},\mathcal{A})$, where $0<p\leq1$ and $\omega$ is a weight on $\mathbb{Z}$. If $\omega$ is admissible, then the conclusion is trivial since $\ell^p_\omega(\mathbb{Z},\mathcal{A})$ and $\ell^1(\mathbb{Z},\mathcal{A})$ are the required algebras. In fact, the existence of the largest such algebra is always trivial because of the vector-valued analogue of Wiener's theorem. The case of interest is when $\omega$ is not admissible and $f$ is invertible in $\ell^1(\mathbb{Z},\mathcal{A})$, but not in $\ell^p_\omega(\mathbb{Z},\mathcal{A})$. In this case, we know that there exists a weight $\nu$ such that $\nu\leq\omega$, $\nu$ is constant (nonconstant) if  $\omega$ is constant (nonconstant) and $f$ is invertible in $\ell^p_\nu(\mathbb{Z},\mathcal{A})$ (see Theorem~\ref{thm1}); and in fact, there are infinitely many such weights. Let $\nu'$ be another such weight with $\nu\leq \nu'\leq\omega$. Then we have the following chain of inclusions, from Lemma~\ref{lem:weighted Fourier algebras inclusions},
\begin{align*}
    \ell^p_\omega(\mathbb{Z},\mathcal{A}) \subset \ell^p_{\nu'}(\mathbb{Z},\mathcal{A}) \subset \ell^p_\nu(\mathbb{Z},\mathcal{A}) \subset \ell^1(\mathbb{Z},\mathcal{A})
\end{align*}
with infinitely many such algebras in between. Specifically, the question is to find the smallest algebra of the form $\ell^p_\eta(\mathbb{Z},\mathcal{A})$ in the above chain that contains an inverse of $f$ because if $p<1$  and $q\in(p,1)$, then $f$ will have an inverse in $\ell^p(\mathbb{Z},\mathcal{A}) \subset \ell^q(\mathbb{Z},\mathcal{A})$ by Theorem~\ref{thm:Wiener-Domar-Zelazko}, and so it makes sense to fix the $p$ and work on the weights $\nu$ to find the smallest such algebra, which is further justified by the case of $p=1$ as in that case we have no other option than to change the weight. Due to it, the question can now be equivalently stated as whether the weight $\nu$ obtained above can be maximized and minimized. The above problem of finding the smallest algebra corresponds to the maximizing case, while the minimizing case is related to finding the largest such algebra for which the answer is the trivial weight $\nu\equiv1$ as seen before. The case of maximizing or obtaining the smallest algebra is not obvious and is not true in general. Here, we find this maximized weight that settles the above questions in the sense that its corresponding algebra may not be the smallest algebra that contains the inverse of $f$ but it will be the largest in the above chain that does not contain an inverse of $f$. A similar result is also obtained for the continuous case, Theorem~\ref{thr}.

%%%%%%%%%%%%%%%%%%%%%%%%%%%%%%%%%%%%%%%%%%%%%%%%%%
The paper is structured as follows. In Section~\ref{sec:pre}, we establish the notation and provide the requisite definitions, along with some known results. The technical lemmas required for our analysis are provided in Section~\ref{sec:lemmas}. Section~\ref{sec:CLBS} is devoted to Wiener type theorems for countable projective and inductive limits of (quasi-)Beurling algebras, namely, algebras of Type-I and Type-II, and this section is further divided into three subsections: \ref{subsec:d=1}, \ref{subsec:d=2} and \ref{subsec:d=infty} to deal with the case of one dimension, finite dimension, and infinite dimension, respectively. In Section~\ref{sec:application}, we show that vector-valued spaces of rapidly decreasing sequences and exponentially decreasing sequences defined in quasi setting are inverse-closed and obtain a hierarchy of inverse-closed algebras. The results on maximization of the weights are presented in Section~\ref{sec:Maximizing weight} for both the discrete and continuous case. In Section~\ref{sec:concluding remarks}, we establish Wiener type theorems for algebras of Type-I and Type-II in one dimension with $\{p_n\}_{n\in\mathbb{N}}\subset(1,\infty)$, followed by concluding remarks and some open questions.

\section{Preliminaries} \label{sec:pre}

First, we set the notation and give some basic definitions. 

Here, $\mathbb{Z}$ denotes the set of integers, $\mathbb{N}$ denotes the set of positive integers, $\mathbb{R}$ denotes the set of real numbers, $\mathbb{C}$ denotes the set of complex numbers, $\mathbb{T}=\{z\in\mathbb{C}:|z|=1\}$ denotes the unit circle in $\mathbb{C}$, $B(z,\epsilon)=\{z'\in\mathbb{C}:|z-z'|<\epsilon\}$ denotes the open ball in $\mathbb{C}$ with center $z$ and radius $\epsilon>0$, and for $0\leq a\leq b,\ \Gamma(a,b)=\{z\in\mathbb{C}:a\leq |z|\leq b\}$ denotes the closed annulus. 
For $d\in\mathbb{N}$, $\mathbb{X}^d$ denotes the $d$-copies of a set $\mathbb{X}$, that is, $\mathbb{X}^d=\{(x_1,x_2,\dots,x_d):x_i\in \mathbb{X} \,\ (1\leq i\leq d)\}$. 
%The symbol $\mathcal{X}_{1}\hookrightarrow \mathcal{X}_{2}$ denotes the continuous embedding of the topological linear space $\mathcal{X}_{1}$ into $\mathcal{X}_{2}$.

The following is an important inequality which will be used frequently.
\begin{align} \label{eq:ineqp<1}
    \text{If } 0<p\leq1 \text{ and } a,b\geq0, \text{ then } (a+b)^p\leq a^p +b^p.
\end{align}

\begin{definition}\cite{Ze}
Let $0<p\leq1$, and let $\mathcal{A}$ be an algebra. 
    \begin{itemize}
        \item A mapping $\|\cdot\| : \mathcal{A} \to [0,\infty)$ is a \emph{$p$-norm} on $\mathcal{A}$ if the following conditions hold for all $x,y\in \mathcal{A}$ and $\alpha \in \mathbb{C}$.
        \begin{enumerate}
        \item $\|x\|=0$ if and only if $x=0$.
        \item $\|x+y\|\leq\|x\|+\|y\|$.
        \item $\|\alpha x\|= |\alpha|^p \|x\|$.
        \item $\|xy\|\leq\|x\|\|y\|$.
    \end{enumerate}
    
    \item The algebra $\mathcal{A}$ along with a $p$-norm $\|\cdot\|$ is a \emph{$p$-normed algebra}. 
    
    \item If $\mathcal{A}$ is complete in the $p$-norm, then $(\mathcal{A},\|\cdot\|)$ is a \emph{$p$-Banach algebra}. 
    
    \item When $p=1$, the map $\|\cdot\|$ is a \emph{norm} on $\mathcal{A}$, and $(\mathcal{A},\|\cdot\|)$ is a \emph{normed 
    algebra} or a \emph{Banach algebra} if it is complete. 
    
    \item When $0<p<1$, the $p$-norm and $p$-normed (Banach) algebra are also known as \emph{quasi-norm} and \emph{quasi-normed (Banach) algebra}, respectively.
\end{itemize}
\end{definition}

\begin{definition}\cite{Kan}
Let $\mathcal{A}$ be an algebra without unit element, and let $\mathcal{A}_\mathbf{1}=\mathcal{A}\times\mathbb{C}=\{(a,\alpha):a\in\mathcal{A}, \alpha\in\mathbb{C}\}$. 
Then $\mathcal{A}_\mathbf{1}$ is a unital algebra with the operations defined as follows:
\begin{enumerate}
\item $(a,\alpha)+(b,\beta)=(a+b,\alpha+\beta)$,
\item $\beta(a,\alpha)=(\beta a,\beta\alpha)$,
\item $(a,\alpha)(b,\beta)=(ab+\alpha b+\beta a,\alpha\beta)$,
\end{enumerate}
for all $a,b\in\mathcal{A}$ and $\alpha,\beta\in\mathbb{C}$. The algebra $\mathcal{A}_\mathbf{1}$ is the \emph{unitisation} of $\mathcal{A}$ and the process is known as \emph{adjoining a unit} to $\mathcal{A}$. The element $\mathbf{1}=(0,1)$ is the unit element of $\mathcal{A}_\mathbf{1}$. The element $(a,\alpha)$ of $\mathcal{A}_\mathbf{1}$ is also written as $a+\lambda\mathbf{1}$ or $\lambda\mathbf{1}+a$.
\end{definition}

Throughout the paper, we fix the following unless otherwise stated.
\begin{enumerate}
    \item $\mathcal{A}$ is a unital Banach algebra. 
    \item $\|\cdot\|$ denotes the norm of $\mathcal{A}$.
    \item $1_\mathcal{A}$ denotes the unit or identity of $\mathcal{A}$, that is, $a1_\mathcal{A}=a=1_\mathcal{A}a$ for all $a\in\mathcal{A}$.
    \item All algebras considered are complex.
    \item $d$ is a positive integer.
\end{enumerate}

\subsection{Weights}
Let $\mathbb{X}\in\{\mathbb{Z}^d,\mathbb{Z^\mathbb{N},\mathbb{R}\}}$. A \emph{weight} on $\mathbb{X}$ is a (Borel measurable, if $\mathbb{X}=\mathbb{R}$) map $\omega:\mathbb{X} \to [1,\infty)$ which satisfies 
\begin{align} \label{def:weight}
    \quad \omega(x+y)\leq \omega(x)\omega(y) \quad \text{for all} \quad x,y \in \mathbb{X}.
\end{align}

Next, we define some technical quantities for these weights that facilitate the study of the Gel'fand spaces and the construction of the required weights; see \cite{Bh, Da2, kb2} for details.

\begin{itemize}
    \item For a weight $\omega$ on $\mathbb{Z}$, define 
        \begin{align} \label{def:rho for Z}
            \rho_{1,\omega}=\sup\{\omega(n)^{1/n}:n \leq 0 \} \leq 1 \leq \rho_{2,\omega}=\inf\{\omega(n)^{1/n}:n \geq 0 \}.
        \end{align} 

    \item For a weight $\omega$ on $\mathbb{Z}^2$, define 
        \begin{align} \label{def:rho for Z^2} \nonumber
            \rho_{1,\omega}=\sqrt{\sup_{(m_1,m_2)\in\mathbb{N}\times\mathbb{Z}}\omega(-m_1,m_2)^{-\frac{1}{m_1}}} \leq 1 \leq \rho_{2,\omega}=\sqrt{\inf_{(m_1,m_2)\in\mathbb{N}\times\mathbb{Z}}\omega(m_1,m_2)^{\frac{1}{m_1}}}, \\ 
            \mu_{1,\omega}=\sqrt{\sup_{(m_1,m_2)\in\mathbb{Z}\times\mathbb{N}}\omega(m_1,-m_2)^{-\frac{1}{m_2}}} \leq 1 \leq \mu_{2,\omega}=\sqrt{\inf_{(m_1,m_2)\in\mathbb{Z}\times\mathbb{N}}\omega(m_1,m_2)^{\frac{1}{m_2}}}.
        \end{align}

    \item For a weight $\omega$ on $\mathbb{Z}^\mathbb{N}$ and $i\in\mathbb{N}$, define
        \begin{align} \label{def:rho for Z^infty} \nonumber
            &\rho_{i,\omega}=\sup \left\{\omega(\alpha_1,\alpha_2,\dots,\alpha_{i-1},-\alpha_i,\alpha_{i+1},\dots)^{-\frac{1} {\alpha_i}}:\alpha_i\in\mathbb{N},\alpha_j\in\mathbb{Z}\ (j\neq i) \right\} \\
            \text{and} \quad &\mu_{i,\omega}=\inf \left\{\omega(\alpha_1,\alpha_2,\dots,\alpha_{i-1},\alpha_i,\alpha_{i+1},\dots)^{\frac{1}{\alpha_i}}:\alpha_i\in\mathbb{N},\alpha_j\in\mathbb{Z}\ (j\neq i) \right\}.
        \end{align}
    Then $\rho_{i,\omega}\leq1\leq\mu_{i,\omega}$ for all $i\in\mathbb{N}$. 
\end{itemize}

It is known that a weight $\omega$ on $\mathbb{Z}$ is admissible if and only if $$\rho_{1,\omega}=1=\rho_{2,\omega}.$$

\begin{definition}
For convenience, we adopt the following definitions.
\begin{itemize}
    \item A weight $\omega$ on $\mathbb{Z}^2$ is \emph{admissible} if $\rho_{1,\omega}=\rho_{2,\omega}=\mu_{1,\omega}=\mu_{2,\omega}=1$.

    \item For a weight $\omega$ on $\mathbb{Z}^\mathbb{N}$, define the set 
    \begin{align} \label{def:S_omega}
        S_\omega=\{i\in\mathbb{N}:\mu_{i,\omega} \neq 1 \ \text{or}\ \rho_{i,\omega}\neq1\}.
    \end{align}
    \begin{enumerate}
        \item $\omega$ is admissible if $S_\omega$ is empty, that is, $\mu_{i,\omega} =1=\rho_{i,\omega}$ for all $i\in\mathbb{N}$.

        \item $\omega$ is $m$-nonadmissible for $m\in\mathbb{N}$ if the cardinality of the set $S_\omega$ is $m$.

        \item $\omega$ is $\infty$-nonadmissible if $S_\omega$ is infinite.
    \end{enumerate}
\end{itemize}
\end{definition}

Note that if a weight $\omega$ on $\mathbb{Z}^2$ or $\mathbb{Z^N}$ satisfies the GRS condition, then it is admissible in our sense. But the other way around may not be true. See \cite{DedaniaGoswami2022, Da2} for more details for the case of two variables.

A weight $\omega$ on $\mathbb{Z}^2$ \emph{satisfies the property $\ast_\omega$} \cite{Da2} if any of the following two conditions is satisfied. 
\begin{enumerate}
    \item $\rho_{1,\omega}=\rho_{2,\omega}=\mu_{1,\omega}=\mu_{2,\omega}=1$, and either $\omega(m_1,-m_2)=\omega(m_1,m_2)$ for all $(m_1,m_2)\in\mathbb{Z}^2$ or $\omega(-m_1,m_2)=\omega(m_1,m_2)$ for all $(m_1,m_2)\in\mathbb{Z}^2$.
    
    \item $\sup\{\omega(m_1,m_2)^\frac{1}{m_1}\omega(m_1,-m_2)^\frac{1}{m_1}:m_1<0,m_2\in\mathbb{Z}\}=\inf\{\omega(m_1,m_2)^\frac{1}{m_1}\omega(m_1,-m_2)^\frac{1}{m_1}:m_1\in\mathbb{N},m_2\in\mathbb{Z}\}$ and \\$\sup\{\omega(m_1,m_2)^\frac{1}{m_2}\omega(-m_1,m_2)^\frac{1}{m_2}:n<0,m\in\mathbb{Z}\}=\inf\{\omega(m_1,m_2)^\frac{1}{m_2}\omega(-m_1,m_2)^\frac{1}{m_2}:m_2\in\mathbb{N},m_1\in\mathbb{Z}\}$.
\end{enumerate}

For $n\in\mathbb{N}$, let  $0<r_j\leq 1\leq s_j < \infty$ ($1\leq j \leq n$), and let $i_1<i_2<\dots <i_n$ be positive integers. Define a weight $\omega$ on $\mathbb{Z}^\mathbb{N}$ by 
\begin{align*}
\omega(\alpha)=\prod_{j=1}^n h_j^{\alpha_{i_j}}, \quad  \text{where} \quad 
    h_j=\begin{cases}
    r_j, \ \text{if} \ \alpha_{i_j}\leq 0 \\ s_j, \ \text{if} \ \alpha_{i_j}\geq 0    
    \end{cases}
     \text{and} \,\, \alpha_j\in\mathbb{Z} \ \text{if} \,\, j\notin\{i_1,i_2,\dots,i_l\}.    
\end{align*}
Since $\omega$ depends only on $n$ variables at the $i_1,i_2,\dots,i_n$ positions, we refer to such weights as \emph{weights of order $n$}.

Lastly, for a weight $\omega$ on $\mathbb{R}$, define
\begin{align} \label{eq:rho for R}
    \rho_{1,\omega}=\sup\{\log\omega(x)^{1/x}:x<0\} \leq 0 \leq \rho_{2,\omega}=\inf\{\log\omega(x)^{1/x}:x>0\}.
\end{align} 
A weight $\omega$ on $\mathbb{R}$ is \emph{admissible} if $\rho_{1,\omega}=0=\rho_{2,\omega}$.

\subsection{Quasi-Beurling algebras}
Let $\mathbb{X} \in \{\mathbb{Z}^d,\mathbb{Z}^\mathbb{N},\mathbb{R}\}$ be considered as a group with addition, and let $\mu$ be the Lebesgue measure when $\mathbb{X}=\mathbb{R}$ and the counting measure when $\mathbb{X} \in \{\mathbb{Z}^d,\mathbb{Z}^\mathbb{N}\}$. For $0< p \leq 1$, a weight $\omega$ on $\mathbb{X}$ and a unital Banach algebra $\mathcal{A}$, let 
\begin{align*}
    L^p_\omega(\mathbb{X},\mathcal{A})=\left\{f: \mathbb{X} \to \mathcal{A}: \|f\|_{L^p_\omega}=\int_\mathbb{X} \|f(x)\|^p\omega(x)^p d\mu(x)<\infty \right\}.
\end{align*}
When $\mathbb{X}=\mathbb{R}$, the functions are further supposed to be strongly measurable.
For $f, g \in L^p_\omega(\mathbb{X},\mathcal{A})$, define \emph{convolution} by $$(f\star g)(x)=\int_\mathbb{X} f(x-y)g(y) d\mu(y) \quad (x\in \mathbb{X}).$$ 

The space $L^1_\omega(\mathbb{X},\mathcal{A})$ is always closed under convolution and a Banach algebra with the norm $\|\cdot\|_{L^1_\omega}$. If $0<p<1$, then in \cite[Theorem 2]{De}, it is shown that it is closed under convolution and a $p$-Banach algebra with the $p$-norm $\|\cdot\|_{L^p_\omega}$ if and only if $\mathbb{X}$ is discrete, and so we only consider $\ell^p_\omega(\mathbb{Z}^d,\mathcal{A})$ and $\ell^p_\omega(\mathbb{Z}^\mathbb{N},\mathcal{A})$ when $p<1$. These weighted ($p$-)Banach algebras are known as \textit{($p$-)Beurling algebras} or \textit{(quasi-)Beurling algebras}. Since they are algebras with convolution as a product, they are also referred to as \emph{convolutive algebras}.

Also, note that $\ell^p_\omega(\mathbb{Z}^d,\mathcal{A})$ and $\ell^p_\omega(\mathbb{Z}^\mathbb{N},\mathcal{A})$ are unital algebras with the unit or identity element given by $\delta_0(\mathbf0)=1_\mathcal{A}$ and zero otherwise, where $\mathbf{0}$ is the zero element of $\mathbb{Z}^d$ and $\mathbb{Z}^\mathbb{N}$. But $L^1_\omega(\mathbb{R},\mathcal{A})$ does not have an identity and so we consider the unitisation of $L^1_\omega(\mathbb{R},\mathcal{A})$ by adjoining the unit $\mathbf1$ and denoted by $L^1_\omega(\mathbb{R},\mathcal{A})_\mathbf1$.

For $f\in \ell^1(\mathbb{Z}^d,\mathcal{A})$, the \textit{Fourier transform} of $f$, $\widehat f:\mathbb{T}^d\to\mathcal{A}$, is defined as 
\begin{align*} %\label{def:FT of f in l1(Z^d)}
    \widehat{f}(z)=\sum_{n \in \mathbb{Z}^d} f(n) z^n \quad (z\in\mathbb{T}^d).
\end{align*}
It is well defined as for each $z\in\mathbb{T}^d$,
\begin{align*}
    \|\widehat{f}(z)\| = \left\| \sum_{n \in \mathbb{Z}^d} f(n) z^n \right\| \leq \sum_{n \in \mathbb{Z}^d} \left\|  f(n) \right\| = \|f\|_{\ell^1} <\infty.
\end{align*}
If $0<p\leq1$ and $\omega$ is a weight on $\mathbb{Z}^d$, then we have $\ell^p_\omega(\mathbb{Z}^d,\mathcal{A})\subset \ell^1(\mathbb{Z}^d,\mathcal{A})$ (see Lemma~\ref{lem:weighted Fourier algebras inclusions}) and so the Fourier transform of $f\in\ell^p_\omega(\mathbb{Z}^d,\mathcal{A})$ is well defined. 

Let $\mathbb{T}^\infty=\{(z_1,z_2,z_3,\dots)\in\mathbb{C}^\infty:|z_i|=1\ \text{for all}\ i\in\mathbb{N}\}$ be the countable product of $\mathbb{T}$. For non-zero $\alpha=(\alpha_1,\alpha_2,\alpha_3,\dots)\in\mathbb{Z}^\mathbb{N}$ and $z=(z_1,z_2,z_3,\dots)\in\mathbb{T}^\infty$, define $$z^\alpha=z_1^{\alpha_1}z_2^{\alpha_2}z_3^{\alpha_3}\cdots.$$ Then $z^\mathbf0=1$, and since nonzero $\alpha$ has finite length, this product is well defined. Indeed, for nonzero $\alpha\in\mathbb{Z}^\mathbb{N}$, let $n_\alpha$ be the positive integer such that $\alpha_{n_\alpha}\neq0$ and $\alpha_{k}=0$ for all $k>n_\alpha$. Then
$$z^\alpha=z_1^{\alpha_1}z_2^{\alpha_2}z_3^{\alpha_3} \cdots z_{n_\alpha}^{\alpha_{n_\alpha}} \quad \text{and} \quad |z^\alpha|=|z_1|^{\alpha_1} |z_2|^{\alpha_2} |z_3|^{\alpha_3} \cdots |z_{n_\alpha}|^{\alpha_{n_\alpha}}.$$
If $f\in \ell^1(\mathbb{Z}^\mathbb{N},\mathcal{A})$, then the \textit{Fourier transform} of $f$, $\widehat f:\mathbb{T}^\infty\to\mathcal{A}$, is defined as 
\begin{align*} %\label{def:FT of f in l1(Z^N)}
    \widehat{f}(z)=\sum_{\alpha \in \mathbb{Z}^\mathbb{N}} f(n) z^\alpha \quad (z\in\mathbb{T}^\infty).
\end{align*}
Again, it is well defined for all $f\in\ell^p_\omega(\mathbb{Z}^\mathbb{N},\mathcal{A})\subset\ell^1(\mathbb{Z}^\mathbb{N},\mathcal{A})$ for  $0<p\leq1$ and a weight $\omega$ on $\mathbb{Z}^\mathbb{N}$.

A similar analysis follows for $L^1_\omega(\mathbb{R},\mathcal{A})$ with a weight $\omega$ on $\mathbb{R}$. In fact, let $T_\omega=\{x+iy\in\mathbb{C}:\rho_{1,\omega}\leq x \leq\rho_{2,\omega}, y\in\mathbb{R}\}$, and let $a+ib\in T_\omega$. Then $\sup\{\log\omega(x)^{1/x}:x<0\} \leq a \leq \inf\{\log\omega(x)^{1/x}:x>0\}$ gives $e^{ax}\leq\omega(x)$ for all $x\in\mathbb{R}$. So, for $f\in L^1_\omega(\mathbb{R},\mathcal{A})$, 
$$\int_\mathbb{R} \|f(x)\||e^{(a+ib)x}|dx \leq \int_\mathbb{R} \|f(x)\|e^{ax}dx \leq \int_\mathbb{R} \|f(x)\|\omega(x)dx<\infty.$$ 
Thus, we define the \emph{Fourier transform} $\widehat{f}$ of $f\in L^1_\omega(\mathbb{R},\mathcal{A})$ by $$\widehat{f}(z)=\int_\mathbb{R} f(x) e^{zx}dx \in \mathcal{A} \quad (z\in T_\omega).$$
Observe that the usual definition corresponds to $T_\omega=i\mathbb{R}$, which occurs when $\omega$ is an admissible weight, particularly when $\omega\equiv1$. Here, the continuous analogue, that is, the case of $\mathbb{R}$ is considered only for maximizing the weight in Section~\ref{sec:Maximizing weight} and so we do not consider the multivariate case of $\mathbb{R}$.

Since we work here with $\mathcal{A}$-valued algebras, we extend the definition of inverse-closedness to the vector-valued case by the following definition.

\begin{definition} \label{def:inverse-closed}
Let $\mathbb{X}$ be one of the sets $\mathbb{Z}^d$ or $\mathbb{Z}^\mathbb{N}$. A subalgebra $\mathcal{X}$ of $\ell^1(\mathbb{X},\mathcal{A})$ is said to be \emph{inverse-closed} if every $f\in\mathcal{X}$ that is invertible in $\ell^1(\mathbb{X},\mathcal{A})$ is also invertible in $\mathcal{X}$, or equivalently, if $f\in\mathcal{X}$ is such that $\widehat{f}(z)$ is invertible for all $z$, then $f$ is invertible in $\mathcal{X}$.
\end{definition}

\subsection{Some known results} \label{subsec:known analogues of Wiener's theorem}
First, state the results from \cite{kb,kb2} that summarize most of the known analogues of Wiener's theorem for $0<p\leq1$.

\begin{theorem} \cite[Theorem 1]{kb} \label{thm1}
Let $0<p\leq1$, $\omega$ be a weight on $\mathbb{Z}$, $\mathcal A$ be a unital Banach algebra, and let $f \in \ell^p_\omega(\mathbb{Z},\mathcal A)$. If $\widehat{f}(z)$ is left invertible (respectively, right invertible, invertible) in $\mathcal A$ for all $z \in \mathbb{T}$, then there is a weight $\nu$ on $\mathbb{Z}$ such that $\nu \leq \omega$, $\nu$ is constant if and only if $\omega$ is constant, $\nu$ is admissible if and only if $\omega$ is admissible, and a left inverse (respectively, right inverse, inverse) of $f$ is in $\ell^p_\nu(\mathbb{Z},\mathcal A)$. 
\end{theorem}

\begin{theorem} \cite[Theorem 3.1]{kb2} \label{2var}
Let $0<p\leq1$, $\mathcal{A}$ be a unital Banach algebra, and let $\omega$ be a weight on $\mathbb{Z}^2$ such that either at least one of $\rho_{1,\omega},\rho_{2,\omega},\mu_{1,\omega}$ and $\mu_{2,\omega}$ is not $1$ or $\omega$ satisfies property $\ast_\omega$. Let $f\in \ell^p_\omega(\mathbb{Z}^2,\mathcal{A})$ be such that $\widehat{f}(z_1,z_2)$ is left invertible (respectively, right invertible, invertible) in $\mathcal{A}$ for all $z_1,z_2\in\mathbb{T}$. Then there is a weight $\nu$ on $\mathbb{Z}^2$ such that $\nu\leq\omega$, $\nu$ is constant if and only if $\omega$ is constant, and a left inverse (respectively, right inverse, inverse) of $f$ is in $\ell^p_\nu(\mathbb{Z}^2,\mathcal A)$. 
\end{theorem}

\begin{theorem} \cite[Theorem 4.1 and Theorem 4.2]{kb2} \label{thi}
Let $0<p\leq1$, $\mathcal{A}$ be a unital Banach algebra, $\omega$ be a weight on $\mathbb{Z}^\mathbb{N}$, and let $f\in \ell^p_{\omega}(\mathbb{Z}^\mathbb{N},\mathcal{A})$ be such that $\widehat{f}(z)$ is left invertible (respectively, right invertible, invertible) in $\mathcal{A}$ for all $z\in\mathbb{T}^\infty$. Then there is a weight $\nu$ on $\mathbb{Z}^\mathbb{N}$ such that $\nu\leq\omega$, $\nu$ is constant if and only if $\omega$ is constant, and a left inverse (respectively, right inverse, inverse) of $f$ is in $\ell^p_\nu(\mathbb{Z}^\mathbb{N},\mathcal{A})$.
Furthermore, suppose that the cardinality of the set $S_\omega$ (see \eqref{def:S_omega}) is $k\in\mathbb{N}\cup\{\infty\}$. Then there is a weight $\nu$ of order $l\in\{1,2,\dots,k\}$ if $k\in\mathbb{N}$, or of order $l\in\mathbb{N}$ if $k=\infty$, such that $\nu\leq\omega$, $\nu$ is constant if and only if $\omega$ is constant, and a left inverse of $f$ is in $\ell^p_\nu(\mathbb{Z}^\mathbb{N},\mathcal{A})$. 
\end{theorem}

\begin{theorem} \cite[Theorem 3]{kb} \label{thr}
Let $\mathcal{A}$ be a unital Banach algebra with unit $1_\mathcal{A}$, $\omega$ be a weight on $\mathbb{R}$, and let $f\in L^1_\omega(\mathbb{R},\mathcal{A})$ be such that $1_\mathcal{A}+\widehat{f}(t)$ is left invertible (respectively, right invertible, invertible) in $\mathcal{A}$ for all $t\in\mathbb{R}$. Then there is a weight $\nu$ on $\mathbb{R}$ such that $\nu\leq\omega$, $\nu$ is constant if and only if $\omega$ is constant, $\nu$ is admissible if and only if $\omega$ is admissible, and a left inverse (respectively, right inverse, inverse) of $\mathbf1+f$ is in $L^1_\nu(\mathbb{R},\mathcal{A})$. In particular, if $\omega$ is admissible, then $\mathbf{1}+f$ is left invertible (respectively, right invertible, invertible) in $L^1_\omega(\mathbb{R},\mathcal{A})_\mathbf1$.
\end{theorem}

\begin{remark} \label{rem:inverse closedness if omega is GRS}
It follows from the construction of $\nu$ in the proofs of above results that $\nu=\omega$ whenever $\omega$ satisfies the GRS condition, or equivalently the corresponding algebra is inverse-closed.
\end{remark}

\begin{remark}
    Note that $\mathcal{A}$ is not required to be commutative which means that the corresponding algebras may be noncommutative. Therefore, a left inverse, if it exists, is not required to be unique.
\end{remark}

Lastly, we state some results that will be used here to prove our results.

\begin{lemma} \cite[Lemma 1.2.7 (b)]{Kan} \label{lem:norm less than 1 invertible}
Let $(\mathcal{A},\|\cdot\|)$ be a unital Banach algebra, and let $x\in\mathcal{A}$ be such that $\|1_\mathcal{A}-x\|<1$. Then $x$ is invertible in $\mathcal{A}$.
\end{lemma}

\begin{lemma} \cite[Lemma 1.3.3]{Kan} \label{lem:weights bounded on compact sets}
Let $G$ be a locally compact group, $\omega$ be a weight on $G$, and let $K$ be a compact subset of $G$. Then there are $0<a\leq b<\infty$ such that $a \leq \omega(x) \leq b $ for all $x\in K$, that is, $\omega$ is bounded away and bounded on every compact subsets of $G$.
\end{lemma}

\section{Lemmas} \label{sec:lemmas}
In this section, we provide the lemmas required to prove our results. First, we state a lemma that gives the inclusion of the quasi-Banach spaces under study.

\begin{lemma} \label{lem:weighted Fourier algebras inclusions}
Let $0 < p \leq q \leq 1$, $\mathbb{X}$ be $\mathbb{Z}^d$ or $\mathbb{Z}^\mathbb{N}$, $\omega$ and $\nu$ be weights on $\mathbb{X}$ such that $\nu \leq \omega$, and let $\mathcal{A}$ be a Banach algebra. Then we have the following chain of inclusions
\begin{align*} 
    \ell^{p}_{\omega}(\mathbb{X},\mathcal{A}) \subset \ell^p_{\nu}(\mathbb{X},\mathcal{A}) \subset \ell^q_\nu(\mathbb{X},\mathcal{A}) \subset \ell^1(\mathbb{X},\mathcal{A})    
\end{align*}
 with the norm inequality
\begin{align} \label{eq:inequality of norms}
    \|f\|_{\ell^1(\mathbb{X},\mathcal{A})} \leq \|f\|_{\ell^q_{\nu}(\mathbb{X},\mathcal{A})}^\frac{1}{q} \leq \|f\|_{\ell^p_{\nu}(\mathbb{X},\mathcal{A})}^\frac{1}{p} \leq \|f\|_{\ell^p_{\omega}(\mathbb{X},\mathcal{A})}^\frac{1}{p}.
\end{align}
\end{lemma}
\begin{proof}
Let $f\in\ell^{p}_{\omega}(\mathbb{X},\mathcal{A})$. Then using $1\leq \nu \leq \omega$, and \eqref{eq:ineqp<1} for $0<p\leq q\leq 1$ and $\frac{p}{q}\leq1$, it follows that
\begin{align*}
    \|f\|_{\ell^1(\mathbb{X},\mathcal{A})} = \sum_{\alpha \in \mathbb{X}} \|f(\alpha)\| &\leq  \left( \sum_{\alpha \in \mathbb{X}} \|f(\alpha)\| \nu(\alpha) \right)^\frac{q}{q} \\
    &\leq  \left( \sum_{\alpha \in \mathbb{X}} \|f(\alpha)\|^q \nu(\alpha)^q \right)^\frac{1}{q} = \|f\|_{\ell^q_{\nu}(\mathbb{X},\mathcal{A})}^\frac{1}{q} \\
    & =  \left( \sum_{\alpha \in \mathbb{X}} \|f(\alpha)\|^q \nu(\alpha)^q \right)^\frac{p}{pq}  \\
    &\leq  \left( \sum_{\alpha \in \mathbb{X}} \|f(\alpha)\|^p \nu(\alpha)^p \right)^\frac{1}{p} = \|f\|_{\ell^p_{\nu}(\mathbb{X},\mathcal{A})}^\frac{1}{p} \\
    &\leq  \left( \sum_{\alpha \in \mathbb{X}} \|f(\alpha)\|^p \omega(\alpha)^p \right)^\frac{1}{p} = \|f\|_{\ell^p_{\omega}(\mathbb{X},\mathcal{A})}^\frac{1}{p}. \\
\end{align*}
\end{proof}

The following lemma depends on $\mathcal{A}$ being a complex algebra, which is necessary for the theorems in Section~\ref{subsec:known analogues of Wiener's theorem} to hold.

\begin{lemma} \label{lem:l^p(X,C) subset l^p(X,A)}
Let $0 < p \leq 1$, $\mathbb{X}$ be $\mathbb{Z}^d$ or $\mathbb{Z}^\mathbb{N}$, $\omega$ be a weight on $\mathbb{X}$, and let $\mathcal{A}$ be a unital Banach algebra. Then $\ell^p_\omega(\mathbb{X},\mathbb{C})$ is identified as a unital subalgebra of $\ell^p_\omega(\mathbb{X},\mathcal{A})$, both having the same unit. In particular, $\ell^p_\omega(\mathbb{X},\mathbb{C}) \subset \ell^p_\omega(\mathbb{X},\mathcal{A})$. Moreover, $f \in \ell^p_\omega(\mathbb{X},\mathbb{C})$ is invertible in $\ell^p_\omega(\mathbb{X},\mathbb{C})$ if and only if $f$ is invertible in $\ell^p_\omega(\mathbb{X},\mathcal{A})$.
\end{lemma}
\begin{proof}
Note that $\mathbb{C}$ can be viewed as a subalgebra of $\mathcal{A}$ by the map $\mathbb{C} \ni \alpha \mapsto \alpha 1_\mathcal{A}\in\mathcal{A}$. For $f\in\ell^p_\omega(\mathbb{X},\mathbb{C})$, define $\mathbf{f}(n)=f(n) 1_\mathcal{A}\in\mathcal{A}$ for $n\in \mathbb{X}$. It is clear that $\mathbf{f}\in \ell^p_\omega(\mathbb{X},\mathcal{A})$. Using the map $\ell^p_\omega(\mathbb{X},\mathbb{C}) \ni f \mapsto \mathbf{f} \in \ell^p_\omega(\mathbb{X},\mathcal{A})$, it follows that $\ell^p_\omega(\mathbb{X},\mathbb{C})$ is a unital subalgebra of $\ell^p_\omega(\mathbb{X},\mathcal{A})$ and both have the same unit.

If $f \in \ell^p_\omega(\mathbb{X},\mathbb{C})$ is invertible in $\ell^p_\omega(\mathbb{X},\mathbb{C})$, then clearly it is invertible in $\ell^p_\omega(\mathbb{X},\mathcal{A})$. Let $f \in \ell^p_\omega(\mathbb{X},\mathbb{C})$ be invertible in $\ell^p_\omega(\mathbb{X},\mathcal{A})$ with an inverse $g$. Then $\widehat{f}$, the Fourier transform of $f$, is a complex valued function that is nowhere zero and so $\widehat{g}=1/\widehat{f}$ implies that $g\in \ell^p_\omega(\mathbb{X},\mathbb{C})$.
\end{proof}

We have the following result for inverse-closedness.
\begin{proposition} \label{prop:inverse-closedness}
Let $0 < p \leq 1$, $\mathbb{X}$ be $\mathbb{Z}^d$ or $\mathbb{Z}^\mathbb{N}$, $\omega$ be a weight on $\mathbb{X}$, and let $\mathcal{A}$ be a unital Banach algebra. Then $\ell^p_\omega(\mathbb{X},\mathcal{A})$ is inverse-closed if and only if $\omega$ satisfies the GRS condition.
\end{proposition}
\begin{proof}
The case of $\mathbb{X}=\mathbb{Z}^d$ follows from particular case of Remark \ref{rem:inverse closedness if omega is GRS} and Theorem \ref{thm:Wiener-Domar-Zelazko} along with Lemma \ref{lem:l^p(X,C) subset l^p(X,A)}. The case of $\mathbb{X}=\mathbb{Z}^\mathbb{N}$ follows from Remark \ref{rem:inverse closedness if omega is GRS} and the fact that $\ell^p_\omega(\mathbb{Z}^\mathbb{N},\mathbb{C})$ is inverse-closed if and only if $\omega$ satisfies the GRS condition combined with Lemma \ref{lem:l^p(X,C) subset l^p(X,A)}.
\end{proof}

Next, we state some lemmas describing the relations between the weights in the family when a weight in the family is admissible or not admissible.

\begin{lemma} \label{lem:weights on Z admissible} 
Let $\{\omega_n\}_{n\in\mathbb{N}}$ be a family of weights on $\mathbb{Z}$ such that $\omega_n\leq\omega_{n+1}$ for all $n\in\mathbb{N}$, and let $\{\nu_n\}_{n\in\mathbb{N}}$ be a family of weights on $\mathbb{Z}$ such that $\nu_n\geq\nu_{n+1}$ for all $n\in\mathbb{N}$. 
\begin{enumerate}
    \item \label{lem:increasing weights on Z admissible} If $\omega_k$ is admissible for some $k$, then $\omega_n$ is admissible for all $n\leq k$.
        
    \item \label{lem:increasing weights on Z not admissible} If $\omega_k$ is not admissible for some $k$, then $\omega_n$ is not admissible for all $n\geq k$.

    \item \label{lem:decreasing weights on Z admissible} If $\nu_k$ is admissible for some $k$, then $\nu_n$ is admissible for all $n\geq k$.

    \item \label{lem:decreasing weights on Z not admissible} If $\nu_k$ is not admissible for some $k$, then $\nu_k$ is not admissible for all $n\leq k$.
\end{enumerate}
\end{lemma}
\begin{proof}
For $n\in\mathbb{N}$, $\omega_n(m)\leq\omega_{n+1}(m)$ for all $m\in\mathbb{Z}$ and so
\begin{align*} 
    \omega_{n+1}(-m)^{-\frac{1}{m}} \leq \omega_n(-m)^{-\frac{1}{m}} \quad \text{and} \quad \omega_n(m)^\frac{1}{m}\leq\omega_{n+1}(m)^\frac{1}{m} \quad \text{for all} \quad m\in\mathbb{N}. 
\end{align*}
This implies that, from \eqref{def:rho for Z}, 
\begin{align} \label{eq:rho inequality for increasing weights on Z}
    \rho_{1,\omega_{n+1}}\leq\rho_{1,\omega_n} \leq 1 \leq \rho_{2,\omega_n} \leq \rho_{2,\omega_{n+1}}.
\end{align}
Similarly, since $\nu_{n+1} \leq \nu_n$, we have
\begin{align} \label{eq:rho inequality for decreasing weights on Z}
    \rho_{1,\nu_n} \leq \rho_{1,\nu_{n+1}} \leq 1 \leq \rho_{2,\nu_{n+1}} \leq \rho_{2,\nu_{n}}.
\end{align}

\eqref{lem:increasing weights on Z admissible} 
If $k=1$, there is nothing to prove. Let $k>1$.  Since $\omega_k$ is admissible, $\rho_{1,\omega_k}=1=\rho_{2,\omega_k}$. Taking $n=k-1$ in \eqref{eq:rho inequality for increasing weights on Z}, it follows that $\omega_{k-1}$ is admissible. The result follows by repeating the same argument.

\eqref{lem:increasing weights on Z not admissible} 
Since $\omega_k$ is not admissible, at least one of $\rho_{1,\omega_k}<1$ or $\rho_{2,\omega_k}>1$ holds. Taking $n=k$ in \eqref{eq:rho inequality for increasing weights on Z}, it follows that at least one of $\rho_{1,\omega_{k+1}}<1$ or $\rho_{2,\omega_{k+1}}>1$ holds; equivalently $\omega_{k+1}$ is also not admissible. Again, the result follows by repeated arguments.

\eqref{lem:decreasing weights on Z admissible} and \eqref{lem:decreasing weights on Z not admissible} follows by similar arguments using \eqref{eq:rho inequality for decreasing weights on Z}.

\end{proof}

\begin{lemma} \label{lem:weights on Z^2 admissible} 
Let $\{\omega_n\}_{n\in\mathbb{N}}$ be a family of weights on $\mathbb{Z}^2$ such that $\omega_n\leq\omega_{n+1}$ for all $n\in\mathbb{N}$, and let $\{\nu_n\}_{n\in\mathbb{N}}$ be a family of weights on $\mathbb{Z}^2$ such that $\nu_n\geq\nu_{n+1}$ for all $n\in\mathbb{N}$. 
\begin{enumerate}
    \item \label{lem:increasing weights on Z^2 admissible} If $\omega_k$ is admissible for some $k$, then $\omega_n$ is admissible for all $n\leq k$.
        
    \item \label{lem:increasing weights on Z^2 not admissible} If $\omega_k$ is not admissible for some $k$, then $\omega_n$ is not admissible for all $n\geq k$.

    \item \label{lem:decreasing weights on Z^2 admissible} If $\nu_k$ is admissible for some $k$, then $\nu_n$ is admissible for all $n\geq k$.

    \item \label{lem:decreasing weights on Z^2 not admissible} If $\nu_k$ is not admissible for some $k$, then $\nu_k$ is not admissible for all $n\leq k$.

    % \item \label{lem:increasing weights on Z^2 admissible} If $\omega_k$ is admissible some $k$, then $\omega_n$ is admissible for all $n\leq k$.
        
    % \item \label{lem:increasing weights on Z^2 not admissible} If $\omega_k$ is such that at least one of $\rho_{1,\omega_k}<1$, $1<\rho_{2,\omega_k}$, $\mu_{1,\omega_k}<1$ or $1<\mu_{2,\omega_k}$ holds for some $k$, then the same follows for $\omega_n$ for all $n\geq k$.

    % \item \label{lem:decreasing weights on Z^2 admissible} If $\nu_k$ is such that $\rho_{1,\nu_k}=\rho_{2,\nu_k}=\mu_{1,\nu_k}=\mu_{2,\nu_k}=1$ for some $k$, then the same follows for $\nu_n$ for all $n\geq k$.

    % \item \label{lem:decreasing weights on Z^2 not admissible} If $\nu_k$ is such that at least one of $\rho_{1,\nu_k}<1$, $1<\rho_{2,\nu_k}$, $\mu_{1,\nu_k}<1$ or $1<\mu_{2,\nu_k}$ holds for some $k$, then the same follows for $\nu_n$ for all $n\leq k$.
\end{enumerate}
\end{lemma}
\begin{proof}
Let $n\in\mathbb{N}$. Recall from \eqref{def:rho for Z^2}, 
\begin{align*} 
    \rho_{1,\omega_n}=\sqrt{\sup_{(m_1,m_2)\in\mathbb{N}\times\mathbb{Z}}\omega_n(-m_1,m_2)^{-\frac{1}{m_1}}}, \quad \rho_{2,\omega_n}=\sqrt{\inf_{(m_1,m_2)\in\mathbb{N}\times\mathbb{Z}}\omega_n(m_1,m_2)^{\frac{1}{m_1}}}, \\ 
    \mu_{1,\omega_n}=\sqrt{\sup_{(m_1,m_2)\in\mathbb{Z}\times\mathbb{N}}\omega_n(m_1,-m_2)^{-\frac{1}{m_2}}}, \quad \mu_{2,\omega_n}=\sqrt{\inf_{(m_1,m_2)\in\mathbb{Z}\times\mathbb{N}}\omega_n(m_1,m_2)^{\frac{1}{m_2}}}.
\end{align*}
Since $\omega_n(m_1,m_2)\leq\omega_{n+1}(m_1,m_2)$ for all $(m_1,m_2)\in\mathbb{Z}^2$, we have
\begin{align*} 
\omega_{n+1}(-m_1,m_2)^{-\frac{1}{m_1}} \leq \omega_n(-m_1,m_2)^{-\frac{1}{m_1}}, \quad 
\omega_n(m_1,m_2)^{\frac{1}{m_1}} \leq \omega_{n+1}(m_1,m_2)^{\frac{1}{m_1}}, \\
\omega_{n+1}(m_1,-m_2)^{-\frac{1}{m_2}} \leq \omega_n(m_1,-m_2)^{-\frac{1}{m_2}} , \quad 
\omega_n(m_1,m_2)^{\frac{1}{m_2}} \leq \omega_{n+1}(m_1,m_2)^{\frac{1}{m_2}}.
\end{align*}
This implies that 
\begin{align} \label{eq:rho inequality for increasing weights on Z^2}
    \rho_{1,\omega_{n+1}}\leq\rho_{1,\omega_n} \leq 1 \leq \rho_{2,\omega_n} \leq \rho_{2,\omega_{n+1}} \quad \text{and} \quad \mu_{1,\omega_{n+1}}\leq\mu_{1,\omega_n} \leq 1 \leq \mu_{2,\omega_n} \leq \mu_{2,\omega_{n+1}}.
\end{align}
Similarly, $\nu_{n+1} \leq \nu_n$ gives
\begin{align} \label{eq:rho inequality for decreasing weights on Z^2}
    \rho_{1,\nu_n} \leq \rho_{1,\nu_{n+1}} \leq 1 \leq \rho_{2,\nu_{n+1}} \leq \rho_{2,\nu_n} \quad \text{and} \quad \mu_{1,\nu_n}\leq\mu_{1,\nu_{n+1}} \leq 1 \leq \mu_{2,\nu_{n+1}} \leq \mu_{2,\nu_n}.
\end{align}
Now the arguments are same as that given in previous Lemma~\ref{lem:weights on Z admissible}, using \eqref{eq:rho inequality for increasing weights on Z^2} and \eqref{eq:rho inequality for decreasing weights on Z^2} in place of \eqref{eq:rho inequality for increasing weights on Z} and \eqref{eq:rho inequality for decreasing weights on Z}, respectively.
    
\end{proof}

\begin{lemma} \label{lem:weights on Z^N admissible} 
Let $\{\omega_n\}_{n\in\mathbb{N}}$ be a family of weights on $\mathbb{Z}^\mathbb{N}$ such that $\omega_n\leq\omega_{n+1}$ for all $n\in\mathbb{N}$, and let $\{\nu_n\}_{n\in\mathbb{N}}$ be a family of weights on $\mathbb{Z}$ such that $\nu_n\geq\nu_{n+1}$ for all $n\in\mathbb{N}$. 
\begin{enumerate}
    \item \label{lem:increasing weights on Z^N admissible} If $\omega_k$ is admissible for some $k$, then $\omega_n$ is admissible for all $n\leq k$.
        
    \item \label{lem:increasing weights on Z^N not admissible} If $\omega_k$ is such that at least one of $\rho_{i,\omega_k}<1$, $1<\rho_{i,\omega_k}$, $\mu_{i,\omega_k}<1$ or $1<\mu_{i,\omega_k}$ holds for some $i,k\in\mathbb{N}$, then the same follows for $\omega_n$ for all $n\geq k$. Equivalently, $S_{\omega_k}\subset S_{\omega_n}$ for all $n\geq k$, or if $\omega_k$ is $m$-nonadmissible for some $k\in\mathbb{N}$ and $m\in\mathbb{N}\cup\{\infty\}$, then $\omega_n$ is $m$-nonadmissible for all $n\geq k$.

    \item \label{lem:decreasing weights on Z^N admissible} If $\nu_k$ is admissible for some $k$, then $\nu_n$ is admissible for all $n\geq k$.

    \item \label{lem:decreasing weights on Z^N not admissible} If $\nu_k$ is such that at least one of $\rho_{i,\nu_k}<1$, $1<\rho_{i,\nu_k}$, $\mu_{i,\nu_k}<1$ or $1<\mu_{i,\nu_k}$ holds for some $k$, then the same follows for $\nu_n$ for all $n\leq k$. Equivalently, $S_{\omega_k}\subset S_{\omega_n}$ for all $n\leq k$, or if $\nu_k$ is $m$-nonadmissible for some $k\in\mathbb{N}$ and $m\in\mathbb{N}\cup\{\infty\}$, then $\nu_n$ is $m$-nonadmissible for all $n\leq k$.
\end{enumerate}
\end{lemma}
\begin{proof}
The proof shall be clear now as it essentially follows the same techniques as in last two lemmas. The two necessary observations are as follows.
For each $i,n\in\mathbb{N}$, 
\begin{align} \label{eq:rho inequality for increasing weights on Z^N}
    \rho_{i,\omega_{n+1}}\leq\rho_{i,\omega_n} \leq 1 \leq \mu_{i,\omega_n} \leq \mu_{i,\omega_{n+1}}
\end{align}
and
\begin{align} \label{eq:rho inequality for decreasing weights on Z^N}
    \rho_{i,\nu_n} \leq \rho_{i,\nu_{n+1}} \leq 1 \leq \mu_{i,\nu_{n+1}} \leq \mu_{i,\nu_n}.
\end{align}
\end{proof}

\section{Wiener type theorem for countable limit of quasi-Beurling algebras} \label{sec:CLBS}
Now, we are ready for our main results for countable projective and inductive limits of (quasi-)Beurling algebras. First, we see the case of one dimension.

\subsection{One dimensional case} \label{subsec:d=1}
In this subsection, we provide the results for one dimension, that is, for $d=1$. The following is our first main theorem which is for the projective limit.

\begin{theorem} \label{thm:proj limit d=1}
Let $\displaystyle \mathcal{X}=\bigcap_{n\in\mathbb{N}}\ell^{p_n}_{\omega_n}(\mathbb{Z},\mathcal{A})$ be an algebra of Type-I. If $f\in\mathcal{X}$ is such that $\widehat{f}(z)$ is left invertible (respectively, right invertible, invertible) in $\mathcal{A}$ for all $z\in\mathbb{T}$, then there exists a family of weights $\{\nu_n\}_{n\in\mathbb{N}}$ on $\mathbb{Z}$ such that for all $n\in \mathbb N$,
\begin{enumerate}
    \item $\nu_n\leq\omega_n$,
    \item $\nu_n$ is constant (nonconstant) if $\omega_n$ is constant (nonconstant),
    \item $\nu_n$ is admissible (nonadmissible) if $\omega_n$ is constant (nonadmissible),
    \item $\nu_n\leq\nu_{n+1}$, and 
    \item a left inverse of (respectively, right inverse, inverse) $f$ is in $\bigcap_{n\in\mathbb{N}}\ell^{p_n}_{\nu_n}(\mathbb{Z},\mathcal{A})$, which is an algebra of Type-I.
\end{enumerate} 
In particular, $\mathcal{X}$ is inverse-closed if and only if $\omega_n$ is admissible for all $n\in\mathbb{N}$.
\end{theorem}
\begin{proof}
If $\omega_n$ is admissible for all $n\in\mathbb{N}$, then taking $\nu_n=\omega_n$ ($n\in\mathbb{N}$) suffices due to Theorem~\ref{thm1}. Suppose that not all $\omega_n$ are admissible. Let $k\in\mathbb{N}$ be the first positive integer such that $\omega_k$ is not admissible. Since $\omega_k$ is not admissible, by Theorem~\ref{thm1}, there is a weight $\nu$ on $\mathbb{Z}$ such that $f$ is left invertible in $\ell^{p_k}_{\nu}(\mathbb{Z},\mathcal{A})$. Note that this weight $\nu$ is not unique (cf. Theorem~\ref{thm:max weight discrete}). We therefore construct a weight $\nu_k$ in accordance with Theorem~\ref{thm1} that also fulfills the requisite conditions. Let $z\in\Gamma(\rho_{1,\omega_k},\rho_{2,\omega_k})$, see \eqref{def:rho for Z}. Then $|z|^n\leq\omega_k(n)$ for all $n\in\mathbb{Z}$. Since $f\in\ell^{p_k}_{\omega_k}(\mathbb{Z},\mathcal{A})$, we have $\sum_{n\in \mathbb{Z}}\|f(n)\|^{p_k}\omega_k(n)^{p_k}<\infty$. Using it along with \eqref{eq:ineqp<1} for $p_k\leq 1$, we get 
$$\sum_{n\in\mathbb{Z}}\|f(n)\||z|^n \leq \left( \sum_{n\in\mathbb{Z}} \|f(n)\|^{p_k}\omega_k(n)^{p_k} \right)^\frac{1}{{p_k}} < \infty,$$ 
that is, $\displaystyle \sum_{n\in\mathbb{Z}}f(n)z^n$ converges absolutely for all $z\in\Gamma(\rho_{1,\omega_k},\rho_{2,\omega_k})$.

Define $F:\Gamma(\rho_{1,\omega_k},\rho_{2,\omega_k}) \to \mathcal{A}$ by $$F(z)=\sum_{n \in \mathbb{Z}} f(n)z^n \quad (z \in \Gamma (\rho_{1,\omega_k},\rho_{2,\omega_k})).$$ Then $F$ is continuous and $F(z)=\widehat{f}(z)$ for all $z \in \mathbb T$. It means that $\widehat{f}$ can be extended continuously to $\Gamma (\rho_{1,\omega_k},\rho_{2,\omega_k})$. Take any $z_0 \in \mathbb{T}$. Since $\widehat{f}(z_0)$ is left invertible in $\mathcal{A}$ and the set $U_l$ of all left invertible elements of $\mathcal{A}$ is open, there is $\epsilon>0$ such that the open ball $B(\widehat{f}(z_0),\epsilon)$ is contained in $U_l$. Since $\widehat{f}$ is continuous, there exists $\delta_{z_0}>0$ such that $\widehat{f}(z)\in B(\widehat{f}(z_0),\epsilon)$ for all $z \in B(z_0,\delta_{z_0})\cap \Gamma(\rho_{1,\omega_k},\rho_{2,\omega_k})$, that is, $\widehat{f}(z)\in U_l$ for all $z \in B(z_0,\delta_{z_0})\cap \Gamma(\rho_{1,\omega_k},\rho_{2,\omega_k})$. By compactness of $\mathbb{T}$, we get positive real numbers $r_k,s_k$ such that $\rho_{1,\omega_k} \leq r_k \leq 1 \leq s_k \leq \rho_{2,\omega_k}$ and $\widehat{f}(z)\in U_l$ for all $z \in \Gamma(r_k,s_k)$. Since $\omega_k$ is not admissible, $\rho_{1,\omega_k}<1$ or $\rho_{2,\omega_k}>1$. If $\rho_{1,\omega_k}<1$, then choose $r_k<1$, and if $\rho_{2,\omega_k}>1$, then choose $s_k>1$. Define a weight $\nu_k$ on $\mathbb{Z}$ by 
\begin{align*}
    \nu_k(n)= 
    \begin{cases}
    r_k^n \quad \text{if} \ \ n\leq0, \\ 
    s_k^n \quad \text{if} \ \ n\geq0    
    \end{cases}.
\end{align*} 
Then $\nu_k\leq\omega_k$ and it is one of the weight from Theorem~\ref{thm1}. So, $f$ is left invertible in $\ell^{p_k}_{\nu_k}(\mathbb{Z},\mathcal{A})$. 

By Lemma~\ref{lem:weights on Z admissible} \eqref{lem:increasing weights on Z not admissible}, $\omega_n$ is not admissible for all $n\geq k$. And, by \eqref{eq:rho inequality for increasing weights on Z}, 
$$\rho_{1,\omega_{k+1}}\leq\rho_{1,\omega_k}\leq r_k \leq 1 \leq s_k \leq \rho_{2,\omega_k} \leq \rho_{2,\omega_{k+1}}.$$
So, we may choose $r_{k+1},s_{k+1}$ for $\omega_{k+1}$, as done for $\omega_k$, such that
\begin{align*}
    \rho_{1,\omega_{k+1}} \leq r_{k+1} \leq r_k \leq 1 \leq s_k \leq s_{k+1} \leq \rho_{2,\omega_{k+1}},
\end{align*} and define $\nu_{k+1}$ in the same manner as $\nu_k$. This along with $\nu_k(0)=1=\nu_{k+1}(0)$ gives $\nu_k\leq\nu_{k+1}$. Inductively, define $\nu_n$ for $n\geq k$, and it will follow that $\nu_n\leq\nu_{n+1}$ for all $n\geq k$. Now, if $k=1$, then we are done. If not, then take $\nu_n\equiv 1$ for all $1\leq n<k$. This guarantees that $\nu_{n}\leq\nu_{n+1}$ for $1\leq n<k$. 

Thus, we get a family of weights $\{\nu_n\}_{n\in\mathbb{N}}$ on $\mathbb{Z}$ that satisfies $\nu_n\leq\omega_n$ and $\nu_n\leq\nu_{n+1}$ for all $n\in \mathbb N$. Also, the weights are constructed in a such a way that Theorem~\ref{thm1} ensures that $f$ is left invertible in $\ell^{p_n}_{\nu_n}(\mathbb{Z},\mathcal{A})$ for all $n\in\mathbb{N}$ which in turn implies that $f$ is left invertible in $\bigcap_{n\in\mathbb{N}}\ell^{p_n}_{\nu_n}(\mathbb{Z},\mathcal{A})$, which is an algebra of Type-I. 

The particular case follows from Proposition~\ref{prop:inverse-closedness}.
\end{proof}

The following is our second main theorem and is for the inductive limit.

\begin{theorem} \label{thm:inductive limit d=1}
Let $\displaystyle \mathcal{X}=\bigcup_{n\in\mathbb{N}}\ell^{p_n}_{\omega_n}(\mathbb{Z},\mathcal{A})$ be an algebra of Type-II. Let $f\in\mathcal{X}$ be such that $\widehat{f}(z)$ is left invertible (respectively, right invertible, invertible) in $\mathcal{A}$ for all $z\in\mathbb{T}$. Assume further that if $f\in\ell^{p_j}_{\omega_j}(\mathbb{Z},\mathcal{A})$ for some $j\in\mathbb{N}$ such that $\omega_j$ is not admissible and $\omega_{j+1}$ is admissible, then the sequence $\{p_n\}_{n>j}$ is strictly increasing. Then there exists a family of weights $\{\nu_n\}_{n\in\mathbb{N}}$ on $\mathbb{Z}$ such that for all $n\in \mathbb{N}$,
\begin{enumerate}
    \item $\nu_n\leq\omega_n$, 
    \item $\nu_n$ is constant (nonconstant) if $\omega_n$ is constant (nonconstant), provided that all $\omega_n$ are admissible or none of them are.
    \item $\nu_n$ is admissible (nonadmissible) if $\omega_n$ is constant (nonadmissible),
    \item $\nu_{n+1} \leq \nu_n$, and 
    \item a left inverse (respectively, right inverse, inverse) of $f$ is in $\bigcup_{n\in\mathbb{N}}\ell^{p_n}_{\nu_n}(\mathbb{Z},\mathcal{A})$, which is an algebra of Type-II.
\end{enumerate}
In particular, $\mathcal{X}$ is inverse-closed if and only if $\{\omega_n\}_{n\in \mathbb{N}}$ satisfies the extended GRS condition.
\end{theorem}
\begin{proof}
Let $k$ be the smallest positive integer such that $f\in\ell^{p_k}_{\omega_k}(\mathbb{Z},\mathcal{A})$. For each $n\in\mathbb{N}$, we have $\omega_{k+n} \leq \omega_{k}$ and $p_k \leq p_{k+n}$. Thus, by Lemma~\ref{lem:weighted Fourier algebras inclusions}, $f\in\ell^{p_n}_{\omega_n}(\mathbb{Z},\mathcal{A})$ for all $n \geq k$. If $\omega_k$ is admissible, then, by Lemma~\ref{lem:weights on Z admissible} \eqref{lem:decreasing weights on Z admissible}, $\omega_j$ is admissible for all $j\geq k$. In this case, set $\nu_j=\omega_j$ for all $j\in\mathbb{N}$ and the theorem follows from Theorem~\ref{thm1}. 

If $\omega_k$ is not admissible, then as in Theorem~\ref{thm:proj limit d=1}, there are positive reals $r_k,s_k$ such that $\rho_{1,\omega_k} \leq r_k \leq 1 \leq s_k \leq \rho_{2,\omega_k}$ and $\widehat{f}(z)$ is left invertible for all $z\in\Gamma(r_k,s_k)$ with at least one of $r_k$ or $s_k$ not $1$. Again, define the weight $\nu_k$ on $\mathbb{Z}$ by 
\begin{align*}
    \nu_k(m)= 
    \begin{cases}
    r_k^m \quad \text{if} \ \ m \leq 0, \\ 
    s_k^m \quad \text{if} \ \ m \geq 0    
    \end{cases}.
\end{align*} 
Then $\nu_k\leq\omega_k$ and $f$ is left invertible in $\ell^{p_k}_{\nu_k}(\mathbb{Z},\mathcal{A})$. Suppose $\omega_{k+1}$ is also not admissible. Since $\omega_{k+1}\leq\omega_k$, using \eqref{eq:rho inequality for decreasing weights on Z}, we may choose $r_{k+1}$ and $s_{k+1}$, where at least one of them is not $1$, such that  
\begin{align*}
\rho_{1,\omega_k} \leq \rho_{1,\omega_{k+1}} \leq r_{k+1} \leq 1 \leq s_{k+1} \leq \rho_{2,\omega_{k+1}} \leq \rho_{2,\omega_{k}}, 
\end{align*}
along with 
\begin{align*}
    r_k \leq r_{k+1} \leq 1 \leq s_{k+1} \leq s_k,
\end{align*}
and construct $\nu_{k+1}$ accordingly. This implies that $\nu_{k+1}\leq\nu_k$. If all $\omega_j$ are not admissible for $j\geq k$, then construct $\nu_j$ inductively. Otherwise, let $j>k$ be the smallest integer such that $\omega_j$ is admissible and set $\nu_l\equiv1$ for all $l\geq j$. For $l<k$, set $\nu_l=\omega_l$. It is clear from the construction of $\nu_l$ $(l\geq k)$ that 
\begin{align*}
    \nu_l\leq\omega_l \quad \text{for all} \quad l\in\mathbb{N}, \quad \text{and} \quad \nu_{l+1} \leq \nu_l \quad \text{for all} \quad l\in\mathbb{N}\setminus\{k-1\}.
\end{align*}
Note that $\nu_k\leq\omega_k\leq\omega_{k-1}=\nu_{k-1}$. This completes the construction of the family of weights $\{\nu_n\}$. Consider the algebra $\displaystyle \mathcal{X}'=\bigcup_{n\in\mathbb{N}}\ell^{p_n}_{\nu_n}(\mathbb{Z},\mathcal{A})$. Then $\mathcal{X}'$ is an algebra of Type-II and $f$ is left invertible in $\mathcal{X}'$.

The particular case follows from Theorem~\ref{thm:Fageot et. al. Type-II} and Lemma~\ref{lem:l^p(X,C) subset l^p(X,A)}.
\end{proof}

\begin{remark} \label{rem:remarks for X' to be algebra of Type-II} Some remarks for $\mathcal{X}'$ in above Theorem \ref{thm:inductive limit d=1}.
\begin{enumerate}
    \item If we consider the algebra $\mathcal{X}'=\cup_{n=1}^k\ell^{p_n}_{\nu_n}(\mathbb{Z},\mathcal{A})$ in the last theorem, then too a left inverse of $f$ is in $\mathcal{X}'$ but then $\mathcal{X}'$ is nothing but $\ell^{p_k}_{\nu_k}(\mathbb{Z},\mathcal{A})$ which is not an algebra of Type-II.

    \item If $\{p_n\}_{n\in\mathbb{N}}$ is strictly increasing, then taking $\mathcal{X}'=\cup_{n\in\mathbb{N}}\ell^{p_n}_{\nu_n}(\mathbb{Z},\mathcal{A})$ with $\nu_n\equiv1$ for $n>k$ will also suffice, but it violates the condition (ii) and (iii).
    
    \item The further assumption is taken to make $\mathcal{X}'$ an algebra of Type-II as we need to set $\nu_l\equiv1$ for $l\geq j$ when there is some $j>k$ such that $\omega_j$ is admissible but $\omega_k$ is not. If it is dropped, then $\mathcal{X}'$ might become a finite union and hence a single space. It can be further substituted by taking $\{p_n\}_{n\geq j}$ such that it contains a strictly increasing subsequence, but we took $\{p_n\}_{n\geq j}$ to be strictly increasing for simplicity.
\end{enumerate}

\end{remark}     

\begin{remark} \label{rem:interesting case all weight not admissible}
In view of Lemma~\ref{lem:weights on Z admissible} \eqref{lem:decreasing weights on Z admissible}, for analysis on weights, it follows that the case of all $\omega_n$ not admissible is of more interest for algebra of Type-II and it is further justified by the construction of $\{\nu_n\}$ in the above theorem. Moreover, it is clear that $\{\omega_n\}_{n\in\mathbb{N}}$ satisfies the extended GRS condition if and only if $\{\nu_n\}_{n\in\mathbb{N}}$ satisfies the extended GRS condition.
\end{remark}

\subsection{Finite dimensional case} \label{subsec:d=2}
We shall only deal with the case of $d=2$ as the case of arbitrary $d\in\mathbb{N}\setminus\{1,2\}$ can be modified accordingly. It should be noted that it differs from the case of $d=1$ as we need to impose some additional conditions on the weights here.

\begin{theorem} \label{thm:proj limit d=2}
Let $\displaystyle \mathcal{X}=\bigcap_{n\in\mathbb{N}}\ell^{p_n}_{\omega_n}(\mathbb{Z}^2,\mathcal{A})$ be an algebra of Type-I. Suppose that, for each $n\in\mathbb{N}$, $\omega_n$ either satisfies the property $\ast_{\omega_n}$ or at least one of $\rho_{1,\omega_n},\rho_{2,\omega_n},\mu_{1,\omega_n}$ and $\mu_{2,\omega_n}$ is not $1$. If $f\in\mathcal{X}$ is such that $\widehat{f}(z)$ is left invertible (respectively, right invertible, invertible) in $\mathcal{A}$ for all $z\in\mathbb{T}^2$, then there exists a family of weights $\{\nu_n\}_{n\in\mathbb{N}}$ on $\mathbb{Z}^2$ such that 
\begin{enumerate}
    \item $\nu_n\leq\omega_n$ for all $n\in \mathbb{N}$,
    \item $\nu_n\leq\nu_{n+1}$ for all $n\in \mathbb{N}$, and 
    \item a left inverse (respectively, right inverse, inverse) of $f$ is in $\bigcap_{n\in\mathbb{N}}\ell^{p_n}_{\nu_n}(\mathbb{Z}^2,\mathcal{A})$, which is an algebra of Type-I.
\end{enumerate} 
\end{theorem}
\begin{proof}
Let $k\in\mathbb{N}$ be the smallest positive integer such that at least one of $\rho_{1,\omega_k},\rho_{2,\omega_k},\mu_{1,\omega_k}$ and $\mu_{2,\omega_k}$ is not $1$. Set $\nu_n\equiv 1$ for all $1\leq n < k$.
Now, if $(z,\eta)\in \Gamma_{\omega_k}=\Gamma(\rho_{1,\omega_k},\rho_{2,\omega_k}) \times \Gamma(\mu_{1,\omega_k},\mu_{2,\omega_k})$, then $\rho_{1,\omega_k}\leq|z|\leq\rho_{2,\omega_k}$ and $\mu_{1,\omega_k}\leq |\eta| \leq \mu_{2,\omega_k}$ and it implies that $|z|^{2m_1}\leq \omega_k(m_1,m_2)$ and $|\eta|^{2m_2}\leq \omega_k(m_1,m_2)$ for all $(m_1,m_2)\in\mathbb{Z}^2$. So, $|z|^{m_1}|\eta|^{m_2}\leq\omega_k(m_1,m_2)$ for all $(m_1,m_2)\in\mathbb{Z}^2$. Since $f\in\ell^{p_k}_{\omega_k}(\mathbb{Z}^2,\mathcal{A})$ and $p_k<1$, it follows using \eqref{eq:ineqp<1} that 
\begin{align*} 
\sum_{(m_1,m_2)\in\mathbb{Z}^2}\|f(m_1,m_2)\| |z|^{m_1}|\eta|^{m_2} 
&\leq \sum_{(m_1,m_2)\in F}\|f(m_1,m_2)\| \omega_k(m_1,m_2) \\
&\leq \left( \sum_{(m_1,m_2)\in\mathbb{Z}^2}\|f(m_1,m_2)\|^{p_k}\omega_k(m_1,m_2)^{p_k} \right)^\frac{1}{p_k} <\infty. 
\end{align*} 
It means $\sum_{(m_1,m_2)\in\mathbb{Z}^2}\|f(m_1,m_2)\| |z|^{m_1}|\eta|^{m_2}$ converges for all $(z,\eta)\in \Gamma_{\omega_k}$ and so $\widehat f$ can be extended to $\Gamma_{\omega_k}$. Since $\widehat f$ is continuous and $\widehat f(z,\eta)$ is left invertible for all $(z,\eta)\in\mathbb{T}^2$, there are positive reals $r_{k_1}, r_{k_2}, s_{k_1}, s_{k_2}$ such that
\begin{align*}
    \rho_{1,\omega_k}\leq r_{k_1} \leq 1 \leq r_{k_2} \leq \rho_{2,\omega_k}, \quad \mu_{1,\omega_k}\leq s_{k_1} \leq 1 \leq s_{k_2} \leq \mu_{2,\omega_k},
\end{align*}
and $\widehat f(z,\eta)$ is left invertible for all $(z,\eta)\in \Gamma(r_{k_1},r_{k_2}) \times \Gamma(s_{k_1},s_{k_2})$. Notice that if $\rho_{1,\omega_k}<1$, then we get $\rho_{1,\omega_k}\leq r_{k_1} <1$, and if $\rho_{2,\omega_k}>1$, then we get $1<r_{k_2}\leq\rho_{2,\omega_k}$. The same goes for $s_{k_1}$ and $s_{k_2}$. Define $\nu_k:\mathbb{Z}^2\to[1,\infty)$ by
\begin{align*}
\nu_k(m_1,m_2)=\begin{cases} 
r_{k_1}^m s_{k_1}^n & \text{if} \ \  m\leq0, n\leq0 \\ 
r_{k_1}^m s_{k_2}^n & \text{if} \ \ m\leq0, n>0 \\ 
r_{k_2}^m s_{k_1}^n & \text{if} \ \ m>0, n\leq0 \\ 
r_{k_2}^m s_{k_2}^n & \text{if} \ \ m>0, n>0 \end{cases}.    
\end{align*}
Now, $\omega_k\leq\omega_{k+1}$ implies there are positive reals $r_{(k+1)_1}, r_{(k+1)_2}, s_{(k+1)_1}, s_{(k+1)_2}$ similarly as above along with
\begin{align*}
    &\rho_{1,\omega_{k+1}} \leq r_{(k+1)_1} \leq r_{k_1} \leq 1 \leq r_{k_2} \leq r_{(k+1)_2} \leq \rho_{2,\omega_{k+1}} \\
    \text{and} \quad &\mu_{1,\omega_{k+1}}\leq  s_{(k+1)_1} \leq s_{k_1} \leq 1 \leq s_{k_2} \leq s_{(k+1)_2} \leq \mu_{2,\omega_{k+1}}.
\end{align*}
And construct $\nu_{k+1}$ as $\nu_k$ using $r_{(k+1)_1}, r_{(k+1)_2}, s_{(k+1)_1}, s_{(k+1)_2}$. Rest of the argument is similar to the one given in Theorem~\ref{thm:inductive limit d=1}.  This concludes the proof.
\end{proof}

\begin{theorem} \label{thm:inductive limit d=2}
Let $\displaystyle \mathcal{X}=\bigcup_{n\in\mathbb{N}}\ell^{p_n}_{\omega_n}(\mathbb{Z}^2,\mathcal{A})$ be an algebra of Type-II. Suppose that, for each $n\in\mathbb{N}$, $\omega_n$ either satisfies property $\ast_{\omega_n}$ or at least one of $\rho_{1,\omega_n},\rho_{2,\omega_n},\mu_{1,\omega_n}$ and $\mu_{2,\omega_n}$ is not $1$. Let $f\in\mathcal{X}$ be such that $\widehat{f}(z)$ is left invertible (respectively, right invertible, invertible) in $\mathcal{A}$ for all $z\in\mathbb{T}^2$. Assume further that if $f\in\ell^{p_j}_{\omega_j}(\mathbb{Z},\mathcal{A})$ for some $j\in\mathbb{N}$ such that $\omega_j$ is not admissible and $\omega_{j+1}$ is admissible, then the sequence $\{p_n\}_{n>j}$ is strictly increasing. Then there exists a family of weights $\{\nu_n\}_{n\in\mathbb{N}}$ on $\mathbb{Z}^2$ such that 
\begin{enumerate}
    \item $\nu_n\leq\omega_n$ for all $n\in \mathbb{N}$, 
    \item $\nu_{n+1} \leq \nu_n$ for all $n\in \mathbb{N}$ and 
    \item a left inverse (respectively, right inverse, inverse) of $f$ is in $\bigcup_{n\in\mathbb{N}}\ell^{p_n}_{\nu_n}(\mathbb{Z}^2,\mathcal{A})$, which is an algebra of Type-II.
\end{enumerate}
\end{theorem}
\begin{proof}
The proof follows using techniques of Theorem~\ref{thm:inductive limit d=1} along with the weights constructed as in Theorem~\ref{thm:proj limit d=2}.
\end{proof}

\subsection{Infinite dimensional case} \label{subsec:d=infty}
Lastly, in this subsection we provide the theorems for $\mathbb{Z}^\mathbb{N}$. Again, we will just provide the construction of weights $\nu_n$ for non admissible weights $\omega_n$ as rest of the details can be verified as per previous cases.

\begin{theorem} \label{thm:proj limit d=infty}
Let $\mathcal{X}=\bigcap_{n\in\mathbb{N}}\ell^{p_n}_{\omega_n}(\mathbb{Z}^\mathbb{N},\omega_n,\mathcal{A})$ be an algebra of Type-I. If $f\in\mathcal{X}$ is such that $\widehat{f}(z)$ is left invertible (respectively, right invertible, invertible) in $\mathcal{A}$ for all $z\in\mathbb{T}^\infty$, then, for each $n\in\mathbb{N}$, there there exists a family of weights $\{\nu_n\}_{n\in\mathbb{N}}$ on $\mathbb{Z}^\mathbb{N}$ such that 
\begin{enumerate}
    \item $\nu_n\leq\omega_n$ for all $n\in \mathbb{N}$,
    \item $\nu_n\leq\nu_{n+1}$ for all $n\in \mathbb{N}$, and 
    \item a left inverse (respectively, right inverse, inverse) of $f$ is in $\bigcap_{n\in\mathbb{N}}\ell^{p_n}_{\nu_n}(\mathbb{Z}^\mathbb{N},\mathcal{A})$, which is an algebra of Type-I.
\end{enumerate} 
Suppose that $n\in\mathbb{N}$ is the first integer such that $\omega_n$ is not an admissible weight. If for some $m\in\mathbb{N}$, $\omega_n$ is $m$-nonadmissible, then given any $l\in\{1,2,\dots,m\}$, the family $\{\nu_n\}$ can be constructed so that $\nu_j$ ($j\geq n$) are of order $l$; and if $\omega_n$ is $\infty$-nonadmissible, then given any $l\in\mathbb{N}$ the family $\{\nu_n\}$ can be constructed so that $\nu_j$ ($j\geq n$) are of order $l$.
\end{theorem}
\begin{proof}
If $\omega_n$ is such that $\mu_{i,\omega_n}=1=\rho_{i,\omega_n}$ for all $n,i\in\mathbb{N}$, then just take $\nu_n=\omega_n$. Let $k\in\mathbb{N}$ be the first integer such that $\omega_k$ is $m$-nonadmissible for some $m\in\mathbb{N}$. Then $S_{\omega_k}=\{i_1,i_2,\dots,i_m\}$, where $i_1<i_2<\dots<i_m$ are positive integers, and at least one of $\mu_{i_t,\omega_k}$ and $\rho_{i_t,\omega_k}$ is not $1$ for all $1\leq t \leq m$. Fix $l\in\{1,2,\dots,m\}$. 
From \eqref{def:rho for Z^infty}, we have
\begin{align*}
    &\mu_{i,\omega_k}^\frac{1}{l}=\left(\sup\{\omega(\alpha_1,\alpha_2,\dots,\alpha_{i-1},-\alpha_i,\alpha_{i+1},\dots)^{-\frac{1}{\alpha_i}}:\alpha_i\in\mathbb{N},\alpha_j\in\mathbb{Z}\ (j\neq i)\}\right)^\frac{1}{l}, \\
    \text{and} \quad &\rho_{i,\omega_k}^\frac{1}{l}=\left(\inf\{\omega(\alpha_1,\alpha_2,\dots,\alpha_{i-1},\alpha_i,\alpha_{i+1},\dots)^{\frac{1}{\alpha_i}}:\alpha_i\in\mathbb{N},\alpha_j\in\mathbb{Z}\ (j\neq i)\}\right)^\frac{1}{l}.
\end{align*}
Let $z\in D=\Pi_{n\in\mathbb{N}} A_n$, where $A_{i_t}=\Gamma(\mu_{i_t}^\frac{1}{l},\rho_{i_t}^\frac{1}{l})$ and $A_n=\mathbb{T} \ (n\neq i_t)$ for $1\leq t\leq l$. Then $|z_{i_t}|^{l\alpha_{i_t}}\leq\omega(\alpha)$ for $1\leq t \leq l$ and so $|z_{i_1}|^{\alpha_{i_1}}|z_{i_2}|^{\alpha_{i_2}}\cdots|z_{i_l}|^{\alpha_{i_l}} \leq \omega_k(\alpha)$ for all $\alpha=(\alpha_1,\alpha_2,\alpha_3,\dots)\in\mathbb{Z}^\mathbb{N}$. 
As $f\in \ell^{p_k}_{\omega_k}(\mathbb{Z}^\mathbb{N},\mathcal{A})$, $\sum_{\alpha\in\mathbb{Z}^\mathbb{N}} \|f(\alpha)\|^{p_k}\omega_k(\alpha)^{p_k} <\infty$. This implies that 
\begin{align*} 
\sum_{\alpha\in\mathbb{Z}^\mathbb{N}} \|f(\alpha)\||z^\alpha| 
&\leq \sum_{\alpha\in\mathbb{Z}^\mathbb{N}} \|f(\alpha)\||z_{i_1}|^{\alpha_{i_1}}|z_{i_2}|^{\alpha_{i_2}}\cdots|z_{i_l}|^{\alpha_{i_l}} \\ 
& \leq \left( \sum_{\alpha\in\mathbb{Z}^\mathbb{N}} \|f(\alpha)\|^{p_k}\omega_k(\alpha)^{p_k} \right)^\frac{1}{p_k} <\infty. \end{align*}So, $\sum_{\alpha\in\mathbb{Z}^\mathbb{N}} f(\alpha)z^\alpha$ converges for all $z\in D$. 
Since $\widehat f$ is continuous, $\widehat f(z)$ is left invertible for all $z\in\mathbb{T}^\infty$ and set of left invertible elements is open, there are $r_{t,k},s_{t,k}$ $(1\leq t\leq l)$ such that $\mu_{i_t,\omega_k}\leq r_{t,k} \leq 1 \leq s_{t,k} \leq \rho_{i_t,\omega_k}$ and $\widehat f(z)$ is left invertible for all $z\in \Gamma=\Pi_{n\in\mathbb{N}} A_n$, where $A_{i_t}=\Gamma(r_{t,k},s_{t,k})$ and $A_n=\mathbb{T}$ if $n\neq i_t$ for $1\leq t\leq l$. If $\mu_{i_t,\omega_k}<1$, then $r_{t,k}<1$, and if $\rho_{i_t,\omega_k}>1$, then $s_{t,k}>1$. Define the weight $\nu_k$ on $\mathbb{Z}^\mathbb{N}$ as follows
\begin{align*}
    \nu_k(\alpha)= \prod_{t=1}^l h_{t,k}^{\alpha_{i_t}}, \quad  \text{where} \quad 
    h_{t,k}=\begin{cases}
    r_{t,k}, \ \text{if} \ \alpha_{i_t}\leq 0 \\ 
    s_{t,k}, \ \text{if} \ \alpha_{i_t}\geq 0    
    \end{cases}
    \text{and} \,\, \alpha_j\in\mathbb{Z} \ \text{if} \,\, j\notin\{i_1,i_2,\dots,i_l\}.
\end{align*}
Then $\nu_k$ is a weight of order $l$ on $\mathbb{Z}^\mathbb{N}$. For $1\leq t\leq l$, using \eqref{eq:rho inequality for increasing weights on Z^N}, we get the positive reals $r_{t,k+1}$ and $s_{t,k+1}$ such that 
$$\mu_{i_t,\omega_{k+1}}\leq r_{t,k+1} \leq r_{t,k} \leq 1 \leq s_{t,k} \leq s_{t,k+1} \leq \rho_{i_t,\omega_{k+1}}.$$ 
Inductively, for each $1\leq t\leq l$, we get the sequences of positive reals $\{r_{t,j}\}_{j\geq k}$ and $\{s_{t,j}\}_{j\geq k}$ depending on the subfamily of weights $\{\omega_j\}_{j\geq k}$ such that
\begin{align*}
& r_{t,j+1} \leq r_{t,j} \leq 1 \leq s_{t,j} \leq s_{t,j+1}  \\
\text{and} \quad &\mu_{i_t,\omega_j}\leq r_{t,j} \leq 1 \leq s_{t,j} \leq \rho_{i_t,\omega_j}
\end{align*}
for all $j\geq k$. For $j\geq k$, construct the weight $\nu_{j}$ on $\mathbb{Z}^\mathbb{N}$ by 
\begin{align*}
    \nu_{j}(\alpha)= \prod_{t=1}^l h_{t,j}^{\alpha_{i_t}}, \quad  \text{where} \quad 
    h_{t,j}=\begin{cases}
    r_{t,j}, \ \text{if} \ \alpha_{i_t}\leq 0 \\ 
    s_{t,j}, \ \text{if} \ \alpha_{i_t}\geq 0    
    \end{cases}
    \text{and} \,\, \alpha_i\in\mathbb{Z} \ \text{if} \,\, i\notin\{i_1,i_2,\dots,i_l\}.
\end{align*}
Then by construction it is clear that $\nu_j$ is a weight of order $l$ for all $j\geq k$; and $\nu_n\leq \omega_n$ and $\nu_n\leq\nu_{n+1}$ for all $n\in\mathbb{N}$. Rest of the proof follows goes along the same lines as that of Theorem~\ref{thi}.

The case when $\omega_k$ is $\infty$-nonadmissible can be dealt with in a similar fashion with the only change that we fix $l$ from $\mathbb{N}$ in place of $\{1,2,\dots,m\}$.
\end{proof}

\begin{remark}
In the above theorem, we can not define the family of weights $\{\nu_n\}$ such that $\nu_j$ for $j\geq k$ depends on infinitely many indices as we need to make use of the numbers $\rho_{i,\omega_j}^\frac{1}{l}$ and $\mu_{i,\omega_j}^\frac{1}{l}$ for defining the weights $\nu_j$ ($j\geq k$). So, taking $l=\infty$ will give these numbers to be $1$ which in turn will give all the weights $\nu_j\equiv1$ for $j\geq k$. 
\end{remark}

In view of Remark~\ref{rem:remarks for X' to be algebra of Type-II} and Remark~\ref{rem:interesting case all weight not admissible}, we shall take all $\omega_n$ to be not admissible in the following result. 

\begin{theorem} \label{thm:inductive limit d=infty}
Let $\mathcal{X}=\bigcup_{n\in\mathbb{N}}\ell^{p_n}_{\omega_n}(\mathbb{Z}^\mathbb{N},\mathcal{A})$ be an algebra of Type-II. Assume that $\omega_n$ is not admissible for all $n\in\mathbb{N}$. If $f\in\mathcal{X}$ is such that $\widehat{f}(z)$ is left invertible (respectively, right invertible, invertible) in $\mathcal{A}$ for all $z\in\mathbb{T}^\infty$. Then there exists a family of weights $\{\nu_n\}_{n\in\mathbb{N}}$ on $\mathbb{Z}^\mathbb{N}$ such that
\begin{enumerate}
    \item $\nu_n\leq\omega_n$ for all $n\in \mathbb{N}$,
    \item $\nu_{n+1} \leq \nu_n$ for all $n\in \mathbb{N}$, and 
    \item a left inverse (respectively, right inverse, inverse) of $f$ is in $\bigcup_{n\in\mathbb{N}}\ell^{p_n}_{\nu_n}(\mathbb{Z}^\mathbb{N},\mathcal{A})$, which is an algebra of Type-II.
\end{enumerate}
Moreover, if for some $m\in\mathbb{N}$, all $\omega_n$ are $m$-nonadmissible, then given any $l\in\{1,2,\dots,m\}$, the family $\{\nu_n\}$ can be chosen so that all $\nu_n$ are of order $l$; and if all $\omega_n$ are $\infty$-nonadmissible, then given any $l\in\mathbb{N}$ the family $\{\nu_n\}$ can be chosen so that all $\nu_n$ are of order $l$.
\end{theorem}
\begin{proof}
We just show the cases when all $\omega_n$ are at least $m$-nonadmissible for some $m\in\mathbb{N}$. In that case, from Lemma~\ref{lem:weights on Z^N admissible} \eqref{lem:decreasing weights on Z^N not admissible}, it follows that $\bigcap_{n\in\mathbb{N}} S_{\omega_n}$ has cardinality at least $m$. Let $l\in\{1,2,\dots,m\}$. Choose $i_1<i_2<\dots<i_l$ from $\bigcap_{n\in\mathbb{N}} S_{\omega_n}$. Now for each $1\leq t\leq l$, as done in previous theorem, using \eqref{eq:rho inequality for decreasing weights on Z^N} in place of \eqref{eq:rho inequality for increasing weights on Z^N}, we get the sequences of positive reals $\{r_{t,n}\}_{n\in\mathbb{N}}$ and $\{s_{t,n}\}_{n\in\mathbb{N}}$ such that
\begin{align*}
& r_{t,n} \leq r_{t,n+1} \leq 1 \leq s_{t,n+1} \leq s_{t,n} \\
\text{and} \quad &\mu_{i_t,\omega_n}\leq r_{t,n} \leq 1 \leq s_{t,n} \leq \rho_{i_t,\omega_n} \quad \text{for all} \quad n\in\mathbb{N}.
\end{align*} 
The rest of the proof follows by constructing the weights as in the previous theorem and using arguments similar to the ones given in the proof of Theorem~\ref{thm:inductive limit d=1}.
\end{proof}

\begin{remark} \label{rem:particular not true for higher dimensions}
    The particular cases in one dimension case for admissible weights, Theorem~\ref{thm:proj limit d=1} and Theorem~\ref{thm:inductive limit d=1}, do not follow for the multi-dimensional cases as the definition of admissible weights for weights on $\mathbb{Z}^d$ ($d>2$) and $\mathbb{Z}^\mathbb{N}$ considered here differs from the usual definition with the GRS condition.
\end{remark}

Extend the definition of the extended GRS condition from $\mathbb{Z}^d$ to $\mathbb{Z}^\mathbb{N}$ as follows: a family of weights $\{\omega_n\}_{n\in\mathbb{N}}$ on $\mathbb{Z}^\mathbb{N}$ the extended GRS condition if
\begin{align*} 
    \inf_{n\in\mathbb{N}} \left( \lim_{k\to\infty} \omega_n(km)^\frac{1}{k} \right) =1 \quad \text{for all} \quad m\in\mathbb{Z}^\mathbb{N}.
\end{align*} 
Then we have the following theorem whose proof can be verified easily using Proposition~\ref{prop:inverse-closedness} and techniques of Theorem~\ref{thm:Fageot et. al. Type-II} with Lemma~\ref{lem:l^p(X,C) subset l^p(X,A)}.

\begin{theorem} [Inverse-closedness in higher dimension] \label{thm:inverse-closedness in higher dimension}
\noindent
\begin{enumerate}
    \item \label{1thm:inverse-closedness in higher dimension} The algebras of Type-I: 
    \begin{align*}
        \bigcap_{n\in\mathbb{N}}\ell^{p_n}_{\omega_n}(\mathbb{Z}^d,\mathcal{A}) \quad \text{and} \quad  \bigcap_{n\in\mathbb{N}}\ell^{p_n}_{\omega_n}(\mathbb{Z}^\mathbb{N},\mathcal{A})
    \end{align*}
    are inverse-closed if and only if each $\omega_n$ satisfies the GRS condition for all $n\in\mathbb{N}$.

    \item \label{2thm:inverse-closedness in higher dimension} The algebras of Type-II: 
        \begin{align*}
        \bigcup_{n\in\mathbb{N}}\ell^{p_n}_{\omega_n}(\mathbb{Z}^d,\mathcal{A}) \quad \text{and} \quad  \bigcup_{n\in\mathbb{N}}\ell^{p_n}_{\omega_n}(\mathbb{Z}^\mathbb{N},\mathcal{A})
    \end{align*}
    are inverse-closed if and only if $\{\omega_n\}_{n\in \mathbb{N}}$ satisfies the extended GRS condition.
\end{enumerate}
\end{theorem}
% \begin{proof}
% The proof of \eqref{1thm:inverse-closedness in higher dimension} follows from Theorem~\ref{thm1}, Theorem~\ref{2var} and Theorem~\ref{thi}. The proof of \eqref{2thm:inverse-closedness in higher dimension} follows using Theorem~\ref{thm:Fageot et. al. Type-II} along with Lemma~\ref{lem:l^p(X,C) subset l^p(X,A)}.
% \end{proof}

\section{Hierarchy of inverse-closed algebras in quasi framework} \label{sec:application}
For this section, $\mathbb{X}$ is $\mathbb{Z}^d$ or $\mathbb{Z}^\mathbb{N}$, and we consider the following weights on $\mathbb{X}$.
\begin{enumerate}
    \item Polynomial weights: $\omega_s(m)=(1+|m|)^s$ $(m\in \mathbb{X})$ for some $s\geq0$.
    \item Exponential weights: $\nu_r(m)=e^\frac{|m|}{r+1}$ $(m\in\mathbb{X})$ for some $r\geq0$.
    \item Subexponential weights: $\eta_{a,b}(m)=e^{a|m|^b}$ $(m\in \mathbb{X})$ for some $a>0$ and $0<b<1$.
\end{enumerate}

We have the following lemma already known in the literature, but we still provide proof for completeness.

\begin{lemma}
The weights $\omega_s$ and $\eta_{a,b}$ are admissible for all $s\geq0$, $a>0$ and $0<b<1$. And the family of weights $\{\nu_n\}_{n\in\mathbb{N}}$ satisfies the extended GRS condition.
\end{lemma}
\begin{proof}
For each $m\in \mathbb{X}$,
\begin{align*}
    &\lim_{k\to\infty} \omega_s(km)^\frac{1}{k} = \lim_{k\to\infty} (1+k|m|)^\frac{s}{k} = 1, \\
    &\lim_{k\to\infty} \eta_{a,b}(km)^\frac{1}{k} = \lim_{k\to\infty} e^\frac{a|km|^b}{k} = \lim_{k\to\infty} e^\frac{a|m|^b}{k^{1-b}} = 1,
\end{align*}
and
\begin{align*}
    \lim_{n\to\infty} \left( \lim_{k\to\infty} \nu_n(km)^\frac{1}{k} \right) = \lim_{n\to\infty} \left( \lim_{k\to\infty} \left(e^\frac{|km|}{(n+1)}\right)^\frac{1}{k} \right) = \lim_{n\to\infty} \left( \lim_{k\to\infty} e^\frac{|m|}{(n+1)} \right) = 1,
\end{align*}
that is, $\inf_{n\in\mathbb{N}} (\lim_{k\to\infty} \nu_n(km)^\frac{1}{k})=1$. Thus, $\omega_s$ and $\eta_{a,b}$ satisfies the GRS condition and the family $\{\nu_n\}_{n\in\mathbb{N}}$ satisfies the extended GRS condition.
\end{proof}

The spaces of rapidly decreasing and exponentially decreasing sequences are characterized as
\begin{align*}
\mathcal{S}(\mathbb{Z}^d) = \bigcap_{n\in\mathbb{N}}\ell^1_{\omega_n}(\mathbb{Z}^d) \quad \text{and} \quad \mathcal{E}(\mathbb{Z}^d) = \bigcup_{n\in\mathbb{N}}\ell^1_{\nu_n}(\mathbb{Z}^d),
\end{align*}
respectively.

Here, we extend this definition of $\mathcal{S}(\mathbb{Z}^d)$ and $\mathcal{E}(\mathbb{Z}^d)$ to vector-valued and using the quasi-Beurling algebras in the following manner.

\begin{definition} \label{def:S(X,A) and E(X,A)}
For a unital Banach algebra $\mathcal{A}$, define
\begin{enumerate}
    \item the space of rapidly decreasing sequences:
    $$\mathcal{S}^{\{p_n\}}(\mathbb{X},\mathcal{A}) = \bigcap_{n\in\mathbb{N}}\ell^{p_n}_{\omega_n}(\mathbb{X},\mathcal{A}),$$ 
    where $\{p_n\}_{n\in\mathbb{N}}$ is a decreasing sequence in $(0,1]$ and $\{\omega_n\}_{n\in\mathbb{N}}$ is a family of polynomial weights on $\mathbb{X}$.

    \item the space of exponentially decreasing sequences:
    $$\mathcal{E}^{\{q_n\}}(\mathbb{X},\mathcal{A})=\bigcup_{n\in\mathbb{N}}\ell^{q_n}_{\nu_n}(\mathbb{X},\mathcal{A})$$ 
    where $\{q_n\}_{n\in\mathbb{N}}$ is an increasing sequence in $(0,1]$ and $\{\nu_n\}_{n\in\mathbb{N}}$ is a family of exponential weights on $\mathbb{X}$.
\end{enumerate}
\end{definition}

Taking $\mathcal{A=\mathbb{C}}$, $\mathbb{X}=\mathbb{Z}^d$ and $p_n=1=q_n$ for all $n\in\mathbb{N}$, we retrive back the usual $\mathcal{S}(\mathbb{Z}^d)$ and $\mathcal{E}(\mathbb{Z}^d)$. It is clear that the spaces $\mathcal{S}^{\{p_n\}}(\mathbb{X},\mathcal{A})$ and $\mathcal{E}^{\{q_n\}}(\mathbb{X},\mathcal{A})$ are algebras of Type-I and Type-II, respectively, as $\omega_n\leq\omega_{n+1}$ and $\nu_{n+1}\leq\nu_n$ for all $n\in\mathbb{N}$. Thus, we have the following theorem as an immediate consequence of Theorem~\ref{thm:inverse-closedness in higher dimension} in view of Definition~\ref{def:inverse-closed}.

\begin{theorem} \label{thm:S(X,A) and E(X,A) ere inverse-cloed} 
The algebras $\mathcal{S}^{\{p_n\}}(\mathbb{X},\mathcal{A})$ and $\mathcal{E}^{\{q_n\}}(\mathbb{X},\mathcal{A})$ are inverse-closed.
\end{theorem}

Next, we have the following theorem for the hierarchy of inverse-closed algebras in quasi framework. 

\begin{theorem}
If the sequences $\{p_n\}$ and $\{q_n\}$ are such that $0<q_n \leq p \leq p_n\leq1$ for all $n\in\mathbb{N}$, then we have the following hierarchy of inverse-closed algebras:
\begin{align*}
\mathcal{E}^{\{q_n\}}(\mathbb{X},\mathcal{A}) 
\subset \ell^{p}_{\eta_{a,b}}(\mathbb{X},\mathcal{A}) 
\subset  \mathcal{S}^{\{p_n\}}(\mathbb{X},\mathcal{A}) 
\subset \ell^{p_{n}}_{\omega_{s_{n}}}(\mathbb{X},\mathcal{A}) 
\subset \ell^{p_1}_{\omega_{s_1}}(\mathbb{X},\mathcal{A}) 
\subset \ell^1(\mathbb{X},\mathcal{A}),
\end{align*}    
for all $n\in\mathbb{N}$ with $n-1<s_n\leq n$.
\end{theorem}
\begin{proof}
Let $n\in\mathbb{N}$. Then $\eta_{a,b}(m)=e^{a|m|^b} \leq e^\frac{|m|}{(n+1)} = \nu_n(m)$ for sufficiently large $|m|$, that is, there is some $M_n\in\mathbb{N}$ such that $\eta_{a,b}(m) \leq \nu_n(m)$ for all $|m|>M_n$. By Lemma~\ref{lem:weights bounded on compact sets}, there is some $K_n>0$ such that $\eta_{a,b}(m) \leq K_n\nu_n(m)$ for all $|m|\leq M_n$. So, we have $\eta_{a,b}(m) \leq K_n\nu_n(m)$ for all $m\in \mathbb{X}$. Thus, by Lemma~\ref{lem:weighted Fourier algebras inclusions}, $\ell^{q_n}_{\nu_n}(\mathbb{X},\mathcal{A})\subset\ell^{q_n}_{\eta_{a,b}}(\mathbb{X},\mathcal{A})$, and $q_n\leq p$ implies that $\ell^{q_n}_{\eta_{a,b}}(\mathbb{X},\mathcal{A})\subset\ell^p_{\eta_{a,b}}(\mathbb{X},\mathcal{A})$. This gives the inclusion $\mathcal{E}^{\{q_n\}}(\mathbb{X},\mathcal{A}) \subset \ell^p_{\eta_{a,b}}(\mathbb{X},\mathcal{A})$.

For each $n\in\mathbb{N}$, there exists $M_n,K_n>0$ such that $\omega_n(m)=(1+|m|)^n \leq e^{a|m|^b} = \eta_{a,b}(m)$ for $|m|>M_n$, and $\omega_n(m) \leq K_n \eta_{a,b}(m)$ for $|m|\leq M$, which gives $\omega_n\leq N_n \eta_{a,b}$. It along with $p\leq p_n$ and Lemma ~\ref{lem:weighted Fourier algebras inclusions} gives the inclusion $\ell^{p}_{\eta_{a,b}}(\mathbb{X},\mathcal{A}) \subset  \mathcal{S}^{\{p_n\}}(\mathbb{X},\mathcal{A})$.

The rest of the inclusions can be verified using the same techniques.
\end{proof}

\begin{remark}
The Definition~\ref{def:inverse-closed} allowed us to move ahead of $\ell^1(\mathbb{Z}^d)$ as shown in Lemma~\ref{lem:l^p(X,C) subset l^p(X,A)} which says that $\ell^1(\mathbb{Z}^d)$ is itself a inverse-closed algebra contained in $\ell^1(\mathbb{Z},\mathcal{A})$. In particular, for $n>1$, considering $\ell^1(\mathbb{Z}^d,\mathbb{C}^n)$, we get infinitely many algebras which strictly contain $\ell^1(\mathbb{Z}^d)$ and are inverse-closed themselves. And replacing $\mathbb{Z}^d$ by $\mathbb{Z}^\mathbb{N}$, these results still hold and its adds to the novelty.
\end{remark}

\section{Maximizing the weight} \label{sec:Maximizing weight}
In this section, we present the results for the maximization of the weight $\nu$ obtained in Theorem~\ref{thm1} (discrete case) and Theorem~\ref{thr} (continuous case).

For a weight $\omega$ on $\mathbb{Z}$, let $\Gamma_\omega=\Gamma(\rho_{1,\omega},\rho_{2,\omega})$. By $A^\mathrm{o}$ and $\overline{A}$ denote the interior and closure, respectively, of a set $A\subset\mathbb{C}$. The following is our first main theorem in this section which is for the discrete case. 

\begin{theorem} \label{thm:max weight discrete}
Let $0<p\leq1$, $\omega$ be a weight on $\mathbb{Z}$, $\mathcal A$ be a unital Banach algebra, and let $f \in \ell^p_\omega(\mathbb{Z},\mathcal A)$. If $\widehat{f}(z)$ is left invertible (respectively, right invertible, invertible) in $\mathcal A$ for all $z \in \mathbb{T}$, then there is a weight $\eta$ on $\mathbb{Z}$ satisfying the following properties:
\begin{enumerate}
\item $\eta\leq\omega$,
\item $\eta$ is constant (nonconstant) if $\omega$ is constant (nonconstant),
\item $\eta$ is admissible (nonadmissible) if $\omega$ is admissible (nonadmissible),
\item if $\nu$ is a weight on $\mathbb{Z}$ satisfying $\nu \leq M\omega$ for some $M>0$ and a left inverse (respectively, right inverse, inverse) of $f$ is in $\ell^p_\nu(\mathbb{Z},\mathcal A)$, then $\nu\leq K_\nu \eta$ for some constant $K_\nu\geq1$ which depends on $\nu$.
\end{enumerate}
In particular, if $\omega$ is admissible, then $\eta=\omega$, and if $\omega$ is not admissible with $f$ not invertible in $\ell^p_\omega(\mathbb{Z},\mathcal{A})$, then $\ell^p_\eta(\mathbb{Z},\mathcal{A})$ is the largest algebra in the chain $\ell^p_\omega(\mathbb{Z},\mathcal{A})\subset\ell^1(\mathbb{Z},\mathcal{A})$ that does not contain an inverse of $f$.
\end{theorem}
\begin{proof}
As in Theorem~\ref{thm:proj limit d=1}, $\widehat f$ can be extended on $\Gamma(r_1,r_2)$ for some $\rho_{1,\omega}\leq r_1\leq 1 \leq r_2\leq \rho_{2,\omega}$ such that this extension $\widehat{f}(z)$ is left invertible for all $z \in \Gamma(r_1,r_2)$. Let $r$ be the infimum of all such $r_1$ and $s$ be the supremum of all such $r_2$. Then $\rho_{1,\omega}\leq r\leq1\leq s\leq\rho_{2,\omega}$. 

\textit{Case-1:} If $\omega$ is admissible or $\widehat f(z)$ is well defined and left invertible for all $z\in\Gamma_\omega^\mathrm{o}=\{z\in\mathbb{C}:\rho_{1,\omega}<|z|<\rho_{2,\omega}\}$, then $r=\rho_{1,\omega}$ and $s=\rho_{2,\omega}$, and in this case take $\eta=\omega$. 

\textit{Case-2:} Suppose that $\omega$ is not admissible and $\widehat f(z)$ is not invertible for some $z\in\Gamma_\omega^\mathrm{o}$. Then $\rho_{1,\omega}<r$ or $s<\rho_{2,\omega}$. Let 
\begin{align*}
    D=
    \begin{cases}
        \{z\in\mathbb{C}:1\leq |z|< s\}, & \text{if} \quad r=1<s, \\
        \{z\in\mathbb{C}:r<|z|\leq 1\}, & \text{if} \quad r<1=s,   \\ 
        \{z\in\mathbb{C}:r<|z|<s\}, & \text{if} \quad r<1<s.
    \end{cases}
\end{align*} Then in any case $\widehat f(z)$ is left invertible for all $z\in D$ and there is some $z_0\in\overline D$ such that $\widehat f(z_0)$ is not left invertible. Define $\eta : \mathbb{Z} \to[1,\infty)$ as follows:\\
If $\rho_{1,\omega}=r \leq 1<s<\rho_{2,\omega}$, then define 
\begin{align*}
    \eta(n)=\begin{cases} \omega(n), & n\leq 0, \\ s^n, & n>0. \end{cases}
\end{align*}
If $\rho_{1,\omega}< r<1\leq s=\rho_{2,\omega}$, then define 
\begin{align*}
    \eta(n)=\begin{cases} r^n, & n<0, \\ \omega(n), & n\geq0. \end{cases}
\end{align*}
If $\rho_{1,\omega}<r<1<s<\rho_{2,\omega}$, then define 
\begin{align*}
    \eta(n)=\begin{cases} r^n, & n\leq 0, \\ s^n, & n\geq 0. \end{cases}
\end{align*}
Then in any case $\eta$ is a weight on $\mathbb{Z}$ which is constant (admissible) if and only if $\omega$ is constant (admissible), $\eta \leq \omega$, $D\subset\Gamma_\eta$ and $\overline D=\Gamma_\eta$. 

Let $\nu$ be a weight on $\mathbb{Z}$ such that $\nu\leq M\omega$ for some $M>0$ and $\ell^p_\nu(\mathbb{Z},\mathcal A)$ contains a left inverse of $f$. Then it follows that $\rho_{1,\omega}\leq\rho_{1,\nu}\leq 1\leq\rho_{2,\nu}\leq \rho_{2,\omega}$ and $\widehat f(z)$ is left invertible in $\mathcal{A}$ for all $z\in\Gamma_\nu$. So, $\Gamma_\nu\subset D$. By definition of $r,s$ and $\eta$, 
\begin{align*}
    r=\rho_{1,\eta} \leq \rho_{1,\nu} \leq 1 \leq \rho_{2,\nu}\leq \rho_{2,\eta}=s.
\end{align*}
This means that 
\begin{align*}
    \lim_{n\to\infty}\eta(-n)^{-1/n}\leq\lim_{n\to\infty}\nu(-n)^{-1/n} \quad \text{and} \quad \lim_{n\to\infty}\eta(n)^{1/n}\geq\lim_{n\to\infty}\nu(n)^{1/n}.
\end{align*} So, there is $n_0\in\mathbb{N}$ such that $\nu(n)\leq\eta(n)$ for all $|n|>n_0$. Also, there exists some $K_\nu\geq1$ such that $\nu(n)\leq K_\nu \eta(n)$ for all $|n|\leq n_0$. So, $\nu(n)\leq K_\nu \eta(n)$ for all $n\in\mathbb{Z}$. This completes the proof.
\end{proof}

Next, we have the following result for the maximization of weight $\nu$ obtained in Theorem~\ref{thr} and it can be seen as a continuous analogue of Theorems~\ref{thm:max weight discrete} for $p=1$.

\begin{theorem} \label{thm:max weight continuous}
Let $\omega$ be a weight on $\mathbb{R}$, $\mathcal{A}$ be a unital Banach algebra with unit element $1_\mathcal{A}$, and let $f\in L^1_\omega(\mathbb{R},\mathcal{A})$. If $1_\mathcal{A}+\widehat f(it)$ is left invertible (respectively, right invertible, invertible) in $\mathcal{A}$ for all $t\in\mathbb{R}$, then there is a weight $\eta$ on $\mathbb{R}$ satisfying the following properties:
\begin{enumerate}
\item $\eta\leq\omega$,
\item $\eta$ is constant (nonconstant) if $\omega$ is constant (nonconstant),
\item $\eta$ is admissible (nonadmissible) if $\omega$ is admissible (nonadmissible), and
\item if $\nu$ is a weight on $\mathbb{R}$ such that $\nu \leq M\omega$ for some $M>0$ and a left inverse (respectively, right inverse, inverse) of $\mathbf1+f$ is in $f\in L^1_\nu(\mathbb{R},\mathcal{A})_\mathbf1$, then $\nu\leq K_\nu \eta$ for some constant $K_\nu\geq1$ which depends on $\nu$.
\end{enumerate}
In particular, if $\omega$ is admissible, then $\eta=\omega$, and if $\omega$ is not admissible with $\mathbf{1}+f$ not invertible in $L^1_\omega(\mathbb{R},\mathcal{A})_\mathbf{1}$, then $L^1_\eta(\mathbb{R},\mathcal{A})$ is the largest algebra in the chain $L^1_\omega(\mathbb{R},\mathcal{A})\subset L^1(\mathbb{R},\mathcal{A})$ that does not contain an inverse of $\mathbf{1}+f$.
\end{theorem}
\begin{proof}
Let $T_\omega=\{x+iy\in\mathbb{C}:\rho_{1,\omega}\leq x \leq\rho_{2,\omega}, y\in\mathbb{R}\}$. Since $f\in L^1_\omega(\mathbb{R},\mathcal{A})$, by an application of the Riemann-Lebesgue lemma, 
\begin{align*}
    \lim_{|z|\to\infty} \|\widehat{f}(z)\|=0.
\end{align*}
This implies that there is some $M>0$ such that $$\|1_\mathcal{A}-(1_\mathcal{A}+\widehat{f}(x+iy))\|=\|\widehat{f}(x+iy)\|<1$$ for all $x+iy\in T_\omega$ with $|y|>M.$ By Lemma~\ref{lem:norm less than 1 invertible}, $1_\mathcal{A}+\widehat{f}(x+iy)$ is left invertible in $\mathcal{A}$ for all $x+iy\in T_\omega$ with $|y|>M.$ Let $y\in\mathbb{R}$ with $|y|\leq M$. Then by hypothesis $1_\mathcal{A}+\widehat{f}(iy)$ is left invertible in $\mathcal{A}$. Since the set of all left invertible elements of $\mathcal{A}$ is open and $1_\mathcal{A}+\widehat{f}$ is continuous in $T_\omega$, there is $\delta_y>0$ such that $1_\mathcal{A}+\widehat{f}(z)$ is left invertible in $\mathcal{A}$ for all $z\in B(iy,\delta_y) \cap T_\omega$. By compactness of $\{iy:|y|\leq M\}$, we get $r_1,r_2\in\mathbb{R}$ such that $\rho_{1,\omega}\leq r_1 \leq 0 \leq r_2 \leq \rho_{2,\omega}$ and $1_\mathcal{A}+\widehat{f}(z)$ is left invertible in $\mathcal{A}$ for all $z=x+iy\in [r_1,r_2] + i [-M,M]$. Thus, we get $r_1,r_2\in\mathbb{R}$ such that $\rho_{1,\omega}\leq r_1 \leq 0 \leq r_2 \leq \rho_{2,\omega}$ and $1_\mathcal{A}+\widehat{f}(z)$ is left invertible in $\mathcal{A}$ for all $z=x+iy\in [r_1,r_2] + i\mathbb R$. 

Let $r$ be the infimum of all such $r_1$, and let $s$ be the supremum of all such $r_2$. Then $\rho_{1,\omega}\leq r\leq 0 \leq s\leq\rho_{2,\omega}$. 

\textit{Case-1:} If $\omega$ is admissible or $1_\mathcal{A}+\widehat f(z)$ is left invertible for all $z\in T_\omega^\mathrm{o}=\{x+iy\in\mathbb{C}:\rho_{1,\omega}<x<\rho_{2,\omega}\}$, then $r=\rho_{1,\omega}$ and $s=\rho_{2,\omega}$. Take $\eta=\omega$. 

\textit{Case-2:} Suppose that $\omega$ is not admissible and $1_\mathcal{A}+\widehat f(z)$ is not invertible for some $z\in T_\omega^\mathrm{o}$. Then $\rho_{1,\omega}<r$ or $s<\rho_{2,\omega}$. Take 
\begin{align*}
    D=
    \begin{cases}
        \{x+iy\in\mathbb{C}:0\leq x< s\}, & \text{if} \quad r=0<s,\\
        \{x+iy\in\mathbb{C}:r<x\leq 0\}, & \text{if} \quad r<0=s,\\
        \{x+iy\in\mathbb{C}:r<x<s\}, & \text{if} \quad r<0<s.
    \end{cases}
\end{align*}
 Define $\eta:\mathbb{R}\to [1,\infty)$ as follows:\\
If $\rho_{1,\omega}=r\leq0<s<\rho_{2,\omega}$, then $$\eta(x)=\begin{cases} \omega(x) & x\leq0 \\ e^{r_2x} & x>0 \end{cases}.$$
If $\rho_{1,\omega}<r<0\leq s=\rho_{2,\omega}$, then $$\eta(x)=\begin{cases} e^{r_1x} & x\leq0 \\ \omega(x) & x>0 \end{cases}.$$
If $\rho_{1,\omega}<r<0<s<\rho_{2,\omega}$, then $$\eta(x)=\begin{cases} e^{r_1x} & x\leq0 \\ e^{r_2x} & x>0 \end{cases}.$$
Then in any case $\eta$ is a weight on $\mathbb{R}$ which is constant (admissible) if and only if $\omega$ is constant (admissible), $\eta \leq \omega$, $D\subset T_\eta$ and $\overline D=T_\eta$. 

Let $\nu$ be a weight on $\mathbb{R}$ such that $\nu\leq M\omega$ for some $M>0$ and $L^1_\nu(\mathbb{R},\mathcal{A})_\mathbf1$ contains a left inverse of $\mathbf1+f$. Then $\rho_{1,\omega}\leq\rho_{1,\nu}\leq1\leq\rho_{2,\nu}\leq \rho_{2,\omega}$ and $1_\mathcal{A}+\widehat f(z)$ is left invertible in $\mathcal{A}$ for all $z\in S_\nu$. So, $S_\nu\subset D$. Also, 
\begin{align*}
    r=\rho_{2,\eta} \leq \rho_{1,\nu} \leq 1 \leq \rho_{2,\nu}\leq s=\rho_{2,\eta}
\end{align*} 
implies that 
\begin{align*}
    \lim_{x\to\infty}\eta(-x)^{-1/x}\leq\lim_{x\to\infty}\nu(-x)^{-1/x} \quad \text{and} \quad \lim_{x\to\infty}\eta(x)^{1/x}\geq\lim_{x\to\infty}\nu(x)^{1/x}.
\end{align*}
So, there is some $\delta>0$ such that $\nu(x)\leq\eta(x)$ for all $|x|>\delta$. By Lemma~\ref{lem:weights bounded on compact sets}, $\eta$ and $\nu$ are bounded on $[-\delta,\delta]$. Thus, there is $K_\nu\geq1$ such that $\nu(x)\leq K_\nu \eta(x)$ for all $|x|\leq \delta$. So, $\nu(x)\leq K_\nu \eta(x)$ for all $x\in\mathbb{R}$. This concludes the proof.
\end{proof}

\begin{remark} The following explains why, in certain cases, we are unable to obtain the smallest algebra in the given chain that contains an inverse of the given function.
\begin{enumerate}
    \item In Theorem~\ref{thm:max weight discrete}, $f$ may not have an inverse in $\ell^p_\eta(\mathbb{Z},\mathcal{A})$. For example, consider the function $f\in\ell^1_\omega(\mathbb{Z})$ given by $f(0)=2$, $f(1)=-1$ and zero elsewhere, where $\omega(n)=e^{|n|}$ $(n\in\mathbb{Z})$. Its Fourier transform will be $\widehat{f}(z)=2-z$ for $z\in\mathbb{T}$. Then the maximized weight $\eta$ on $\mathbb{Z}$ will be defined as $\eta(n)=2^n$ $(n\geq0)$ and $\eta(n)=e^{-n}$ $(n\leq0)$. Note that $\frac{1}{f}$ is not in $\ell^1_\eta(\mathbb{Z})$.

    \item Similarly, in Theorem~\ref{thm:max weight continuous}, $\mathbf{1}+f$ may not have an inverse in $L^1_\eta(\mathbb{R},\mathcal{A})_\mathbf{1}$. For example, consider the function $f(x)=-e^{-x}$ for $0\leq x \leq 1$ and 0 otherwise in $L^1_\omega(\mathbb{R},\mathbb{C})$ for $\omega(x)=e^{2|x|}$ $(x\in\mathbb{R})$. Then the maximized weight $\eta$ on $\mathbb{R}$ will be defined as $\eta(x)=e^x$ $(x\geq0)$ and $\eta(x)=e^{-2x}$ $(x\leq0)$. Note that the function $\mathbf{1}+f$ is not invertible in $L^1_\eta(\mathbb{R},\mathbb{C})_\mathbf{1}$.
\end{enumerate}
\end{remark}

We state two corollaries for algebras of Type-I and Type-II in one dimension combining Theorem~\ref{thm:proj limit d=1} and Theorem~\ref{thm:inductive limit d=1} with Theorem~\ref{thm:max weight discrete}.

\begin{corollary} \label{cor:Type-I}
Let $\mathcal{X}=\bigcap_{n\in\mathbb{N}}\ell^{p_n}_{\omega_n}(\mathbb{Z},\mathcal{A})$ be an algebra of Type-I, and let $f\in\mathcal{X}$ be such that $\widehat{f}(z)$ is left invertible (respectively, right invertible, invertible) in $\mathcal{A}$ for all $z\in\mathbb{T}$. Further, assume that 
\begin{enumerate}
    \item \label{(i) further assumption Type-I} $\omega_n$ is not admissible for all $n\in\mathbb{N}$  and $f\in\ell^{p_1}_{\omega_1}(\mathbb{Z},\mathcal{A})$; and
    \item \label{(ii) further assumption Type-I} either $\{p_n\}_{n\in\mathbb{N}}$ is strictly increasing or $\{\omega_n\}_{n\in\mathbb{N}}$ is such that $1\neq\rho_{1,\omega_{n+1}}<\rho_{1,\omega_n}$ or $\rho_{2,\omega_n}<\rho_{2,\omega_{n+1}}\neq1$ for all $n\in\mathbb{N}$.
\end{enumerate}
Then there exists a family of weights $\{\eta_n\}_{n\in\mathbb{N}}$ on $\mathbb{Z}$ such that for all $n\in \mathbb N$,
\begin{enumerate}
    \item $\eta_n\leq\omega_n$,
    \item $\eta_n$ is constant (nonconstant) if $\omega_n$ is constant (nonconstant),
    \item $\eta_n$ is admissible (nonadmissible) if $\omega_n$ is admissible (nonadmissible),
    \item $\eta_n\leq\eta_{n+1}$, and 
    \item if $\{\nu_n\}_{n\in\mathbb{N}}$ is a family of weights on $\mathbb{Z}$ such that $\nu_n \leq \omega_n$ and a left inverse (respectively, right inverse, inverse) of $f$ is in $\bigcap_{n\in\mathbb{N}}\ell^{p_n}_{\nu_n}(\mathbb{Z},\mathcal A)$, then $\nu_n\leq K_{\nu_n} \eta_n$ for some constant $K_{\nu_n}\geq1$ which depends on $\nu_n$ for all $n\in\mathbb{N}$.
\end{enumerate} 
\end{corollary}

\begin{corollary} \label{cor:Type-II}
Let $\mathcal{X}=\bigcup_{n\in\mathbb{N}}\ell^{p_n}_{\omega_n}(\mathbb{Z},\mathcal{A})$ be an algebra of Type-II, and let $f\in\mathcal{X}$ be such that $\widehat{f}(z)$ is left invertible (respectively, right invertible, invertible) in $\mathcal{A}$ for all $z\in\mathbb{T}$. Further, assume that 
\begin{enumerate}
    \item \label{(i) further assumption Type-II} $\omega_n$ is not admissible, and $f\in\ell^{p_n}_{\omega_n}(\mathbb{Z},\mathcal{A})$ for all $n\in\mathbb{N}$; and
    \item \label{(ii) further assumption Type-II} either $\{p_n\}_{n\in\mathbb{N}}$ is strictly increasing or $\{\omega_n\}_{n\in\mathbb{N}}$ is such that $\rho_{1,\omega_n}<\rho_{1,\omega_{n+1}}\neq1$ or $1\neq\rho_{2,\omega_{n+1}}<\rho_{2,\omega_n}$ for all $n\in\mathbb{N}$
\end{enumerate}
Then there exists a family of weights $\{\eta_n\}_{n\in\mathbb{N}}$ on $\mathbb{Z}$ such that for all $n\in \mathbb{N}$,
\begin{enumerate}
    \item $\eta_n\leq\omega_n$, 
    \item $\eta_n$ is constant (nonconstant) if $\omega_n$ is constant (nonconstant),
    \item $\eta_n$ is admissible (nonadmissible) if $\omega_n$ is admissible (nonadmissible),
    \item $\eta_{n+1} \leq \eta_n$, and 
    \item if $\{\nu_n\}_{n\in\mathbb{N}}$ is a family of weights on $\mathbb{Z}$ such that $\nu_n \leq \omega_n$ and a left inverse (respectively, right inverse, inverse) of $f$ is in $\bigcup_{n\in\mathbb{N}}\ell^{p_n}_{\nu_n}(\mathbb{Z},\mathcal A)$, then $\nu_n\leq K_{\nu_n} \eta_n$ for some constant $K_{\nu_n}\geq1$ which depends on $\nu_n$ for all $n\in\mathbb{N}$.
\end{enumerate}
\end{corollary}

\begin{remark}
In further assumptions for the above corollaries, the first cannot be dropped as they are necessary for the existence of $\{\eta_n\}_{n\in\mathbb{N}}$ satisfying the required conditions; and the second ensures that all $\ell^{p_n}_{\eta_n}(\mathbb{Z},\mathcal{A})$ are distinct. Together, both ensure that $\bigcap_{n\in\mathbb{N}}\ell^{p_n}_{\eta_n}(\mathbb{Z},\mathcal{A})$ and $\bigcup_{n\in\mathbb{N}}\ell^{p_n}_{\eta_n}(\mathbb{Z},\mathcal{A})$ are indeed the algebras of Type-I and Type-II, respectively, acting like the largest algebra in the corresponding chain of algebras that does not contain an inverse of given $f$.
\end{remark}

\section{Non-quasi analogues and concluding remarks} \label{sec:concluding remarks}
We briefly mention the case of $p>1$ in one dimension. Let $1<p<\infty$, and let $p'$ be its conjugate index, that is, $\frac{1}{p}+\frac{1}{p'}=1$. A weight $\omega$ on $\mathbb{Z}$ is a \emph{$p$-almost monotone algebra weight} \cite{Da} if the following conditions hold:
\begin{enumerate}
\item \label{def:(i) for p-almost algebra weight} $\omega^{-p'}\star\omega^{-p'}\leq\omega^{-p'}$ and $\sum_{n \in \mathbb{Z}} \omega(n)^{-p'}<\infty$.
\item \label{def:(ii) for p-almost algebra weight} If $\rho_{1,\omega}=1$, then there is a positive constant $K$ such that $\omega(n)\leq K\omega(n+k)$ for all $-n,-k \in \mathbb{N}$.
\item \label{def:(iii) for p-almost algebra weight} If $\rho_{2,\omega}=1$, then there is a positive constant $K$ such that $\omega(n)\leq K\omega(n+k)$ for all $n,k \in \mathbb{N}_0$.
\end{enumerate}
The first condition in \eqref{def:(i) for p-almost algebra weight} is required to ensure that $\ell^p_\omega(\mathbb{Z,\mathcal{A}})$ is an algebra (see \cite{Ku}), while the second implies the inclusion $\ell^p_\omega(\mathbb{Z,\mathcal{A}})\subset\ell^1(\mathbb{Z},\mathcal{A})$, which in turn allows the Fourier transform to be defined for $f\in\ell^p(\mathbb{Z,\mathcal{A}})$. And conditions \eqref{def:(ii) for p-almost algebra weight} and \eqref{def:(iii) for p-almost algebra weight} are required for the construction of $p$-almost monotone algebra weight $\nu$ in the following theorem and hence in the consecutive theorems.

\begin{theorem} \cite[Theorem 2]{kb} \label{thm2}
Let $1<p<\infty$, $\omega$ be a $p$-almost monotone algebra weight on $\mathbb Z$, $\mathcal A$ be a unital Banach algebra, and let $f \in \ell^p_\omega(\mathbb{Z},\mathcal A)$. If $\widehat{f}(z)$ is left invertible (respectively, right invertible, invertible) in $\mathcal A$ for all $z \in \mathbb{T}$, then there exist a $p$-almost monotone algebra weight $\nu$ on $\mathbb{Z}$ such that $\nu \leq \omega$, $\nu$ is admissible if and only if $\omega$ is admissible and a left inverse (respectively, right inverse, inverse) of $f$ is in $\ell^p_\nu(\mathbb{Z},\mathcal A)$. In particular, if $\omega$ is an admissible weight, then a left inverse (respectively, right inverse, inverse) of $f$ is in $\ell^p_\omega(\mathbb{Z},\mathcal A)$.
\end{theorem}

Using this result, together with the techniques used in Theorem~\ref{thm:proj limit d=1} and Theorem~\ref{thm:inductive limit d=1}, we obtain the following two theorems. Since constant weights, particularly the trivial weight $\omega\equiv1$, are not $p$-almost monotone algebra weights for any $p>1$, we impose an additional hypothesis. We provide only the constructions of the weights, as the remaining details can be verified easily. 

\begin{theorem} \label{thm:p>1, proj limit d=1}
Let $\displaystyle \mathcal{X}=\bigcap_{n\in\mathbb{N}}\ell^{p_n}_{\omega_n}(\mathbb{Z},\mathcal{A})$ be an algebra of Type-I, where $\{p_n\}_{n\in\mathbb{N}}$ is a decreasing sequence in $(1,\infty)$ and $\omega_n$ is a $p_n$-almost monotone algebra weight for all $n\in\mathbb{N}$. In addition, assume that all $\omega_n$ are admissible or none of them are. If $f\in\mathcal{X}$ is such that $\widehat{f}(z)$ is left invertible (respectively, right invertible, invertible) in $\mathcal{A}$ for all $z\in\mathbb{T}$, then there exists a family of weights $\{\nu_n\}_{n\in\mathbb{N}}$ on $\mathbb{Z}$ such that  for all $n\in \mathbb N$,
\begin{enumerate}
    \item $\nu_n\leq K_n \omega_n$ for some $K_n>1$,
    \item $\nu_n$ is $p_n$-almost monotone algebra weight,
    \item $\nu_n$ is admissible (nonadmissible) if $\omega_n$ is constant (nonadmissible),
    \item $\nu_n\leq\nu_{n+1}$, and 
    \item a left inverse of (respectively, right inverse, inverse) $f$ is in $\bigcap_{n\in\mathbb{N}}\ell^{p_n}_{\nu_n}(\mathbb{Z},\mathcal{A})$, which is an algebra of Type-I.
\end{enumerate} 
Moreover, $\mathcal{X}$ is inverse-closed if and only if $\omega_n$ is an admissible $p_n$-almost monotone algebra weight for all $n\in\mathbb{N}$.
\end{theorem}
\begin{proof}
Fix $\gamma \in (0,1)$. Let $k\in\mathbb{N}$. Then, as in Theorem~\ref{thm:inductive limit d=1}, we may choose positive reals $r_k,s_k$ such that $\rho_{1,\omega_k} \leq r_k \leq 1 \leq s_k \leq \rho_{2,\omega_k}$ and $\widehat{f}(z)$ is left invertible in $\mathcal{A}$ for all $z \in \Gamma(r_k,s_k)$.

We consider the following three cases. 

\textit{Case-1:} If $\rho_{1,\omega_k}=1$ and $\rho_{2,\omega_k}>1$, then, by \cite[Lemma 2]{Da}, there is $K_{k,1}>1$ such that $\omega_k(m)\leq \widetilde{\omega_k}(m) \leq K_{k,1}\omega_k(m)$ for all $m<0$, where $\widetilde{\omega_k}(m)=\max\{\omega_k(l):m \leq l \leq -1\}$ for all $m<0$. Also, there exists $K_{k,2}>1$ such that $e^{|m|^\gamma} \leq K_{k,2}\omega_k(m)$ for all $n \in \mathbb{N}_0$. Set $K_k=\max\{K_{k,1},K_{k,2}\}$ and define 
\begin{align*}
    \nu_k(m)=
    \begin{cases} 
    \widetilde{\omega_k}(m), & m<0, \\ 
    s_k^\frac{m}{2} e^\frac{|m|^\gamma}{2}, & m\geq 0.
    \end{cases}
\end{align*}

\textit{Case-2:} Let $\rho_{1,\omega_k}<1$ and $\rho_{2,\omega_k}=1$. If $\widetilde{\omega_k}(m)=\max\{\omega_k(l):0\leq l \leq n\}$ for all $m\in \mathbb{N}_0$, then, again by \cite[Lemma 2]{Da}, there is $K_{k,1}>0$ such that $\omega_k(m)\leq \widetilde{\omega_k}(m) \leq M_{k,1}\omega_k(m)$ for all $m\in \mathbb{N}_0$. Also, there is a constant $K_{k,2}>0$ such that $e^{|m|^\gamma} \leq K_{k,2}\omega_k(m)$ for all $m<0$. Let $K_k=\max\{K_{k,1},K_{k,2}\}$ and define 
\begin{align*}
    \nu_k(m)=
    \begin{cases} 
    r_k^\frac{m}{2} e^\frac{|m|^\gamma}{2}, & m<0, \\ 
    \widetilde{\omega_k}(m), & m\geq 0.
    \end{cases}
\end{align*}

\textit{Case-3:} If $\rho_{1,\omega_k}<1<\rho_{2,\omega_k}$, then there is some $K_k>1$ such that $e^{|m|^\gamma} \leq K_k\omega_k(m)$ for all $m\in\mathbb{Z}$. Define 
\begin{align*}
    \nu_k(m)=
    \begin{cases} 
    r_k^\frac{m}{2} e^\frac{|m|^\gamma}{2}, & m<0, \\ 
    s_k^\frac{m}{2} e^\frac{|m|^\gamma}{2}, & m\geq 0.
    \end{cases}
\end{align*}
Then, in any case, $\nu_k$ is a $p_k$-almost monotone algebra weight on $\mathbb{Z}$ such that $\nu_k \leq K_k\omega_k$ for some $K_k>0$. By hypothesis all $\omega_n$ are not admissible and so, using techniques of Theorem~\ref{thm:proj limit d=1}, we may choose $r_n,s_n$ such that $r_{n+1}\leq r_n$ and $s_n\leq s_{n+1}$ for all $n\in\mathbb{N}$. Also, $\omega_n \leq \omega_{n+1}$ implies that $\widetilde{\omega_n} \leq \widetilde{\omega}_{n+1}$ for all $n\in\mathbb{N}$. Thus, $\nu_n\leq\nu_{n+1}$ for all $n\in\mathbb{N}$. The theorem follows from Theorem~\ref{thm2} as these weights are constructed in a similar manner.
\end{proof}

\begin{theorem} \label{thm:p>1 inductive limit d=1}
Let $\displaystyle \mathcal{X}=\bigcup_{n\in\mathbb{N}}\ell^{p_n}_{\omega_n}(\mathbb{Z},\mathcal{A})$ be an algebra of Type-II, where $\{p_n\}_{n\in\mathbb{N}}$ is an increasing sequence in $(1,\infty)$ and $\omega_n$ is a $p_n$-almost monotone algebra weight for all $n\in\mathbb{N}$. In addition, assume that all $\omega_n$ are admissible or none of them are. If $f\in\mathcal{X}$ is such that $\widehat{f}(z)$ is left invertible (respectively, right invertible, invertible) in $\mathcal{A}$ for all $z\in\mathbb{T}$, then there exists a family of weights $\{\nu_n\}_{n\in\mathbb{N}}$ on $\mathbb{Z}$ such that for all $n\in \mathbb{N}$,
\begin{enumerate}
    \item $\nu_n\leq\omega_n$, 
    \item $\nu_n$ is $p_n$-almost monotone algebra weight,
    \item $\nu_n$ is admissible (nonadmissible) if $\omega_n$ is constant (nonadmissible),
    \item $\nu_{n+1} \leq \nu_n$, and 
    \item a left inverse (respectively, right inverse, inverse) of $f$ is in $\bigcup_{n\in\mathbb{N}}\ell^{p_n}_{\nu_n}(\mathbb{Z},\mathcal{A})$, which is an algebra of Type-II.
\end{enumerate}
Moreover, $\mathcal{X}$ inverse-closed if and only if the family $\{\omega_n\}_{n\in \mathbb{N}}$ of $p_n$-almost monotone algebra weight satisfies the extended GRS condition \eqref{def:extended GRS condition}.
\end{theorem}

The following are some final remarks and open problems for future research.

\begin{remark}
The reason for not having maximizing results for the multi-dimensional scenario is again due to the adapted definition of admissible weights as mentioned in Remark~\ref{rem:particular not true for higher dimensions}. Indeed, it is because the proof of the versions of Wiener's theorem in Section~\ref{subsec:known analogues of Wiener's theorem} relies on the Gel'fand theory of Banach algebras and the Gel'fand space of $\ell^p_\omega(\mathbb{Z})$ is identified with $\Gamma(\rho_{1,\omega},\rho_{2,\omega})$, while it is not exactly known for higher dimensions, that is, the Gel'fand space of $\ell^p_\omega(\mathbb{Z}^2)$ may be larger than $\Gamma(\rho_{1,\omega},\rho_{2,\omega}) \times \Gamma(\mu_{1,\omega},\mu_{2,\omega})$, and the same follows for infinite dimensional analogue. The identification of the Gel'fand space of an algebra constitutes an important question in its own right. Dedania and Goswami \cite{DedaniaGoswami2022} worked for the identification of the Gel'fand space of $\ell^1_\omega(\mathbb{Z}^2)$ but it turned out to be more complicated than required. It is known that if $\omega$ is a weight on $\mathbb{Z}^d$ (or $\mathbb{Z}^\mathbb{N}$) that satisfies the GRS condition, then it is $\mathbb{T}^d$ (or $\mathbb{T}^\infty$). This naturally leads to the following open questions, which will eventually complete the study presented here.
\end{remark}

\begin{remark}
For $p>1$, the reason for not having maximizing results is that we need almost monotone algebra weights here and the techniques of Theorem~\ref{thm:max weight discrete} fail as the resulted weights may not be an almost monotone algebra weight. Moreover, the reason for not having the multi dimensional analogues is that the Gel'fand space of $\ell^p_\omega(\mathbb{Z}^d)$ is not known for a weight $\omega$ which is not admissible. This proposes the following question.
\end{remark}

\textbf{Question:} Can the Gel'fand space of $\ell^p_\omega(\mathbb{Z}^2)$ (or $\ell^p_\omega(\mathbb{Z}^\mathbb{N})$) be expressed as a product of two (or infinite) annuli when $\omega$ is a weight on $\mathbb{Z}^2$ (or $\mathbb{Z}^\mathbb{N}$) that does not satisfy the GRS condition? 

If the answer to the above question is negative, one may then pose the following question.

\textbf{Question:} For what class of nonadmissible weights, the Gel'fand space of $\ell^p_\omega(\mathbb{Z}^2)$ (or $\ell^p_\omega(\mathbb{Z}^\mathbb{N})$) be expressed as a product of two (or infinite) annuli?

Concerning the above question, it is known that if the weight under consideration factors as a product of weights in each coordinate, then the desired conclusion follows.

\section*{Statements and Declarations}
\textbf{Acknowledgment.} The first author is grateful to the National Board for Higher Mathematics (NBHM), India, for the research grant (02011/39/2025/NBHM(R. P.)/R$\&$D II/16090). The second author gratefully acknowledges the Post Doctoral Fellowship under the ISIRD project 9--551/2023/IITRPR/10229 IIT Ropar. This work was partially supported by the FIST program of the Department of Science and Technology, Government of India, Reference No. SR/FST/MS--I/2018/22(C).

\textbf{Funding.} No funding was received for conducting this study.

\textbf{Competing Interests.} The authors have no relevant financial or non-financial interests to disclose.

\textbf{Author Contributions.} All authors contributed equally.

\textbf{Availability of data and material.} No data were used for this study.

\textbf{Ethical Approval.} Not applicable.

\bibliography{References}

@article{Bal25,
  author = {K\"{o}hldorfer, Lukas and Balazs, Peter},
  title = {{W}iener Pairs of {B}anach Algebras of Operator-Valued Matrices},
  journal = {J. Math. Anal. Appl.},
  volume = {549},
  number = {2},
  year = {2025},
  pages = {Paper No. 129525},
  issn = {0022-247X},
  doi = {10.1016/j.jmaa.2025.129525},
  url = {https://doi.org/10.1016/j.jmaa.2025.129525},
  mrclass = {47A56 (42C15 46H05)}
}

@article{Ba,
  author = {Baskakov, A. G.},
  title = {Asymptotic Estimates for the Entries of the Matrices of Inverse Operators and Harmonic Analysis},
  journal = {Sib. Math. J.},
  volume = {38},
  number = {1},
  year = {1997},
  pages = {10--22}
}

@article{Badi,
  author = {Baskakov, A. G.},
  title = {Abstract Harmonic Analysis and Asymptotic Estimates of Elements of Inverse Matrices},
  journal = {Mat. Zh.},
  volume = {52},
  number = {2},
  year = {1992},
  pages = {17--26}
}

@article{De,
  author = {Bhatt, S. J. and Dabhi, P. A. and Dedania, H. V.},
  title = {{B}eurling Algebra Analogues of the Theorems of {W}iener--{L}\'evy--\.{Z}elazko and \.{Z}elazko},
  journal = {Studia Math.},
  volume = {195},
  number = {3},
  year = {2009},
  pages = {219--225}
}

@article{Bh,
  author = {Bhatt, S. J. and Dedania, H. V.},
  title = {{B}eurling Algebra Analogues of the Classical Theorems of {W}iener and {L}\'evy on Absolutely Convergent {F}ourier Series},
  journal = {Proc. Indian Acad. Sci. Math. Sci.},
  volume = {113},
  number = {2},
  year = {2003},
  pages = {179--182}
}

@article{Bo,
  author = {Bochner, S. and Phillips, R. S.},
  title = {Absolutely Convergent {F}ourier Expansions for Non-Commutative Normed Rings},
  journal = {Ann. of Math.},
  volume = {43},
  number = {3},
  year = {1942},
  pages = {409--418}
}

@article{Cuccagna,
author = {Cuccagna, Scipio and Tarulli, Mirko},
title = {On Asymptotic Stability of Standing Waves of Discrete {S}chrödinger Equation in $\mathbb{Z}$},
journal = {SIAM Journal on Mathematical Analysis},
volume = {41},
number = {3},
pages = {861-885},
year = {2009},
doi = {10.1137/080732821},
}

@article{Da,
  author = {Dabhi, P. A.},
  title = {On Weighted $\ell^p$-Convergence of {F}ourier Series: A Variant of the Theorems of {W}iener and {L}\'evy},
  journal = {Adv. Oper. Theory},
  volume = {5},
  year = {2020},
  pages = {1832--1838}
}

@article{Da2,
  author = {Dabhi, P. A.},
  title = {On Two Variable {B}eurling Algebra Analogues of Theorems of {W}iener and {L}\'evy on {F}ourier Series},
  journal = {Proc. Amer. Math. Soc.},
  volume = {150},
  number = {3},
  year = {2022},
  pages = {997--1008}
}

@article{kb,
  author = {Dabhi, P. A. and Solanki, K. B.},
  title = {Vector Valued {B}eurling Algebra Analogues of {W}iener's Theorem},
  journal = {Indian J. Pure Appl. Math.},
  volume = {56},
  year = {2025},
  pages = {448--461}
}

@article{kb2,
  author = {Dabhi, P. A. and Solanki, K. B.},
  title = {Multivariate Vector Valued {B}eurling Algebra Analogues of Theorems of {W}iener and \.{Z}elazko},
  journal = {Commun. Korean Math. Soc.},
  volume = {40},
  number = {3},
  year = {2025},
  pages = {687--697}
}

@article{kb3,
  author = {Dabhi, P. A. and Solanki, K. B.},
  title = {{W}iener Type Theorem for Some Operator Algebras and Elements of {B}anach Algebra},
  journal = {J. Math. Sci.},
  volume={290},
  year = {2025},
  pages = {155--170},
  doi = {10.1007/s10958-025-07577-8}
}

@article{DedaniaGoswami2022,
  author  = {Dedania, H. V. and Goswami, V. N.},
  title   = {The {G}el'fand spaces of discrete {B}eurling algebras on $\mathbb{Z}_+^2$ and $\mathbb{Z}^2$},
  journal = {Italian Journal of Pure and Applied Mathematics},
  number  = {47},
  year    = {2022},
  pages   = {1166--1182},
}

@article{Do,
  author = {Domar, Y.},
  title = {Harmonic Analysis Based on Certain Commutative {B}anach Algebras},
  journal = {Acta Math.},
  volume = {96},
  year = {1956},
  pages = {1--66}
}

@article{FeiWeight1979,
  author = {Feichtinger, Hans G.},
  title = {Gewichtsfunktionen auf Lokalkompakten Gruppen},
  journal = {\"{O}sterreich. Akad. Wiss. Math.-Natur. Kl. Sitzungsber. II},
  volume = {188},
  year = {1979},
  pages = {451--471}
}

@article{Fageotcountablelimits,
  author = {Fageot, J. and Unser, M. and Ward, J. P.},
  title = {Beyond {W}iener's Lemma: Nuclear Convolution Algebras and the Inversion of Digital Filters},
  journal = {J. {F}ourier Anal. Appl.},
  volume = {25},
  year = {2019},
  pages = {2037--2063}
}

@article{FernandezGalbisToft2014,
  author  = {Fern{\'a}ndez, C. and Galbis, A. and Toft, J.},
  title   = {Spectral Properties for Matrix Algebras},
  journal = {Journal of Fourier Analysis and Applications},
  volume  = {20},
  number  = {2},
  pages   = {362--383},
  year    = {2014}
}

@article{FernandezGalbisToft2015,
  author  = {Fern{\'a}ndez, C. and Galbis, A. and Toft, J.},
  title   = {Characterizations of {GRS}-Weights, and Consequences in Time-Frequency Analysis},
  journal = {Journal of Pseudo-Differential Operators and Applications},
  volume  = {3},
  number  = {6},
  pages   = {383--390},
  year    = {2015}
}

@book {GRS,
    AUTHOR = {Gel'fand, I. and Ra\v{i}kov, D. and \v{S}hilov, G.},
    TITLE = {Commutative normed rings},
    PUBLISHER = {Chelsea Publishing Co., New York},
    YEAR = {1964},
    PAGES = {306}
}

@book{grafakos2008modern,
  author    = {Loukas Grafakos},
  title     = {Modern {F}ourier Analysis},
  publisher = {Springer},
  year      = {2008}
}

@article{Gro,
  author = {Gr\"{o}chenig, K.},
  title = {Time-Frequency Analysis of {S}j\"{o}strand's Class},
  journal = {Rev. Mat. Iberoam.},
  volume = {22},
  number = {2},
  year = {2006},
  pages = {703--724}
}

@article{Gr,
  author = {Gr\"{o}chenig, K. and Klotz, A.},
  title = {Noncommutative Approximation: Inverse Closed Subalgebras and Off Diagonal Decay of Matrices},
  journal = {Constr. Approx.},
  volume = {32},
  number = {2},
  year = {2010},
  pages = {429--466}
}

@incollection{GroWeight2007,
  author = {Gr\"{o}chenig, K.},
  title = {Weight Functions in Time-Frequency Analysis},
  booktitle = {Pseudo-Differential Operators: Partial Differential Equations and Time-Frequency Analysis},
  series = {Fields Inst. Commun.},
  volume = {52},
  pages = {343--366},
  publisher = {Amer. Math. Soc., Providence, RI},
  year = {2007},
  mrclass = {42C15 (35S05 46B15 47B37 47G30)},
  url = {https://doi.org/10.1090/fic/052/16}
}

@incollection{Groadmissible,
  author = {Gr\"{o}chenig, K.},
  title = {{W}iener's Lemma: Theme and Variations. An Introduction to Spectral Invariance and Its Applications},
  booktitle = {Four Short Courses on Harmonic Analysis},
  pages = {175--234},
  publisher = {Springer, Berlin},
  year = {2010}
}

@book{Kan,
  author = {Kaniuth, E.},
  title = {A Course in Commutative {B}anach Algebras},
  series = {Graduate Texts in Mathematics},
  volume = {246},
  publisher = {Springer},
  year = {2009}
}

@article{Krishtal2011,
  author = {Krishtal, I.},
  title = {{W}iener's Lemma: Pictures at an Exhibition},
  journal = {Rev. Unión Mat. Argent.},
  volume = {52},
  pages = {61--79},
  year = {2011}
}

@article{Ku,
  author = {Kuznetsova, Yulia and Molitor-Braun, Carine},
  title = {Harmonic Analysis of Weighted {$L^p$}-Algebras},
  journal = {Expo. Math.},
  volume = {30},
  number = {2},
  year = {2012},
  pages = {124--153},
  doi = {10.1016/j.exmath.2012.01.002},
  url = {https://doi.org/10.1016/j.exmath.2012.01.002}
}

@article{PelinovskyStefanov,
    author = {Pelinovsky, D. E. and Stefanov, A.},
    title = {On the spectral theory and dispersive estimates for a discrete {S}chrödinger equation in one dimension},
    journal = {Journal of Mathematical Physics},
    volume = {49},
    number = {11},
    pages = {113501},
    year = {2008},
    month = {11},
    issn = {0022-2488},
    doi = {10.1063/1.3005597},
}

@article{reich2015nonanalytic,
  author = {Maximilian Reich and Michael Reissig and Winfried Sickel},
  title  = {Non‑analytic superposition results on modulation spaces with subexponential weights},
  journal = { Journal of Pseudo-Differential Operators and Applications },
volume={7},
pages={365--409},
  year   = {2016},
urldoi={https://doi.org/10.1007/s11868-016-0148-x},
}

@article{ShinSun2013,
  author = {Shin, C. E. and Sun, Q.},
  title = {{W}iener's Lemma: Localization and Various Approaches},
  journal = {Appl. Math. A J. Chin. Univ.},
  volume = {28},
  number = {4},
  pages = {465--484},
  year = {2013}
}

@article{Sun2007,
  author = {Sun, Q.},
  title = {{W}iener's Lemma for Infinite Matrices},
  journal = {Trans. Amer. Math. Soc.},
  volume = {359},
  number = {7},
  pages = {3099--3123},
  year = {2007}
}

@article{Sun2011,
  author = {Sun, Q.},
  title = {{W}iener's Lemma for Infinite Matrices {II}},
  journal = {Constr. Approx.},
  volume = {34},
  pages = {209--235},
  year = {2011}
}

@article{SUN2005567,
  author = {Sun, Q.},
  title = {Wiener's lemma for infinite matrices with polynomial off-diagonal decay},
  journal = {Comptes Rendus Mathematique},
  volume = {340}, 
  number = {8}, 
  pages = {567--570},
  year = {2005}
}

@article{wi,
  author = {Wiener, N.},
  title = {Tauberian Theorems},
  journal = {Ann. of Math.},
  volume = {33},
  pages = {1--100},
  year = {1932}
}

@article{Zag,
  author = {Zagorodnyuk, A. V. and Mitrofanov, M. A.},
  title = {An Analog of {W}iener's Theorem for Infinite-Dimensional {B}anach Spaces ({R}ussian. {R}ussian Summary)},
  journal = {Mat. Zametki},
  volume = {97},
  number = {2},
  pages = {191--202},
  year = {2015}
}

@book{Ze,
  author = {\.{Z}elazko, W.},
  title = {Selected Topics in Topological Algebras},
  series = {Lecture Notes Series},
  number = {31},
  publisher = {Aarhus Univ.},
  year = {1971}
}

@article{Rauhut2007Coorbit,
  author    = {Holger Rauhut},
  title     = {Coorbit space theory for quasi‑{B}anach spaces},
  journal   = {Studia Mathematica},
  volume    = {180},
  number    = {3},
  pages     = {237--253},
  year      = {2007},
  doi       = {10.4064/sm180-3-4}
}

@article{Rauhut2007Wiener,
  author    = {Holger Rauhut},
  title     = {Wiener amalgam spaces with respect to quasi‑{B}anach spaces},
  journal   = {Colloquium Mathematicum},
  volume    = {109},
  number    = {2},
  pages     = {345--362},
  year      = {2007}
}

@article{BastianoniCordero2022,
  author    = {Federico Bastianoni and Elena Cordero},
  title     = {Quasi‑{B}anach modulation spaces and localization operators on locally compact abelian groups},
  journal   = {Banach Journal of Mathematical Analysis},
  volume    = {16},
  article   = {52},
  pages     = {1--71},
  year      = {2022},
  doi       = {10.1007/s43037-022-00205-6}
}

@article{GrochenigPfeufferToft2024,
  author    = {Karlheinz Gr{\"o}chenig and Christine Pfeuffer and Joachim Toft},
  title     = {Spectral invariance of quasi‑{B}anach algebras of matrices and pseudodifferential operators},
  journal   = {Forum Mathematicum},
  volume    = {36},
  number    = {5},
  pages     = {1201--1224},
  year      = {2024},
  doi       = {10.1515/forum-2023-0212}
}
\bibliographystyle{amsplain}

\end{document}